\documentclass[12pt]{amsart}
\usepackage{amsmath}
\usepackage{amssymb}
\usepackage{epsfig}
\usepackage{graphicx}
\numberwithin{equation}{section}

\input xy
\xyoption{all}

\calclayout
\allowdisplaybreaks[3]

\theoremstyle{plain}
\newtheorem{prop}{Proposition}

\newtheorem{theo}[prop]{Theorem}
\newtheorem{coro}[prop]{Corollary}

\newtheorem{lemm}[prop]{Lemma}

\theoremstyle{definition}
\newtheorem{defi}[prop]{Definition}

\newtheorem{rema}[prop]{Remark}

\newtheorem{exam}[prop]{Example}

\newcommand{\eqto}{\stackrel{\lower1.5pt\hbox{$\scriptstyle\sim\,$}}\to}
\DeclareMathOperator{\Spec}{Spec}
\DeclareMathOperator{\Proj}{Proj}
\DeclareMathOperator{\Bun}{Bun}
\DeclareMathOperator{\chara}{char}

\def\ra{\rightarrow}

\def\cD{{\mathcal D}}

\def\cI{{\mathcal I}}

\def\cL{{\mathcal L}}

\def\cO{{\mathcal O}}

\def\cQ{{\mathcal Q}}

\def\cW{{\mathcal W}}
\def\cX{{\mathcal X}}
\def\cY{{\mathcal Y}}

\def\oM{{\overline{M}}}
\newcommand{\wM}{\widetilde{\mathcal{M}}}
\newcommand{\wpM}{\widetilde{\mathcal{M}}^{+}}

\def\fS{{\mathfrak S}}

\def\bG{{\mathbb G}}
\def\bP{{\mathbb P}}
\def\bQ{{\mathbb Q}}

\def\bZ{{\mathbb Z}}

\def\bC{{\mathbb C}}
\def\F{{\mathbb F}}
\def\Z{{\mathbb Z}}

\def\sD{{\mathsf D}}
\def\sB{{\mathsf B}}
\def\sA{{\mathsf A}}

\def\bA{{\mathbb A}}
\def\bF{{\mathbb F}}

\def\Bl{\mathrm{Bl}}
\def\Pic{\mathrm{Pic}}
\def\End{\mathrm{End}}

\def\Aut{\mathrm{Aut}}
\def\Gr{\mathrm{Gr}}

\def\PGL{\mathrm{PGL}}
\def\Sp{\mathrm{Sp}}
\def\GL{\mathrm{GL}}

\def\Ext{\mathrm{Ext}}
\def\Hom{\mathrm{Hom}}

\def\lim{\mathrm{lim}}

\def\Prym{\mathrm{Prym}}

\def\Sym{\mathrm{Sym}}
\def\IJ{\mathrm{IJ}}
\def\JJ{\mathrm{J}}

\makeatother
\makeatletter

\author{}
\address{}
\email{}

\begin{document}

\title[Moduli of Del Pezzo surfaces]{On the moduli of degree 4 Del Pezzo surfaces}

\date{December 22, 2013}

\dedicatory{To Professor Mukai, with admiration}

\author{Brendan Hassett}
\address{
  Department of Mathematics, MS 136,
  Rice University,
  Houston, TX 77005, USA
}
\email{hassett@rice.edu}
\author{Andrew Kresch}
\address{
  Institut f\"ur Mathematik,
  Universit\"at Z\"urich,
  Winterthurerstrasse 190,
  CH-8057 Z\"urich, Switzerland
}
\email{andrew.kresch@math.uzh.ch}

\author{Yuri Tschinkel}
\address{Courant Institute, New York University, New York, NY 11012, USA}
\address{Simons Foundation, 160 Fifth Avenue, New York, NY 10010, USA}
\email{tschinkel@cims.nyu.edu}

\begin{abstract}
We study irreducibility of 
families of degree 4 Del Pezzo surface fibrations over curves.  
\end{abstract}

\maketitle

\section{Introduction}
\label{intro}

Let $X\subset \bP^4$ be a smooth surface defined by the
intersection of two quadrics over an algebraically closed
field $k$ of characteristic different from $2$.
It is known that $X$ is characterized up to isomorphism by the
degeneracy locus of the
pencil of quadrics containing $X$,
i.e., by the form 
\begin{equation}
\label{eqn.f}
f(u,v)=\det(uP+vQ),
\end{equation}
where $P$ and $Q$ are
symmetric $5\times 5$ matrices whose associated quadratic forms
define $X$.
The two-dimensional space of binary quintic forms with nonvanishing discriminant
up to linear change of variable serves as a moduli space of smooth
Del Pezzo surfaces of degree 4.
Mabuchi and Mukai \cite{MM} studied this
from the perspective of Geometric
Invariant Theory.

Any (nontrivial) family $\cX \to \bP^1$ of degree 4 Del Pezzo surfaces
necessarily contains singular fibers.
Generically, these are Del Pezzo surfaces with a single $\sA_1$-singularity.
So a study of families of degree 4 Del Pezzo surfaces necessarily entails
a moduli problem which admits Del Pezzo surfaces with one $\sA_1$-singularity.

It is known that a general smooth
Del Pezzo surface of degree 4 has automorphism group $(\bZ/2\bZ)^4$.
A notable incongruity with standard moduli problems such as
stable curves of genus $g\ge 2$ is that automorphism groups can
\emph{decrease}, rather than increase, upon specialization.
Indeed, the general $\sA_1$-singular degree 4 Del Pezzo surface
has only 8 automorphisms \cite{hosoh}.

In this paper we study the moduli problem of degree 4 Del Pezzo surfaces.
We relate it to the moduli problem of log general type surfaces, as worked out
by Hacking, Keel, and Tevelev \cite{HKT}, where the above-mentioned incongruity
disappears.
We give an explicit description of families of degree 4 Del Pezzo surfaces over
$\bP^1$, general in the sense of having smooth total space,
fibers with at most one $\sA_1$-singularity, and maximal monodromy of
the lines in smooth fibers.
A single discrete invariant, the \emph{height},
is proportional to the number of singular fibers.
We establish the irreducibility of the space of general families
of given height, with exceptional behavior for a few small heights.

Given a family of Del Pezzo surfaces of degree 4 over a base $T$, the
fiberwise vanishing locus of the form \eqref{eqn.f} determines a
\emph{spectral cover} of degree 5 over $T$.
Assuming the family is sufficiently general, the singular members of the
pencils of quadrics are of nodal type, hence contain two families of planes.
This defines a double cover of the spectral cover.
In the case of a family over $\bP^1$, we get
\begin{equation}
\label{eqn.toweroverP1}
\tilde D\to D\to \bP^1.
\end{equation}
Our approach is to obtain a description of families of Del Pezzo surfaces of degree 4
over $\bP^1$ in terms of the spectral curve $D$ together with tower \eqref{eqn.toweroverP1},
and to relate the moduli problem for families to the moduli problem for such
towers.
The machinery of moduli of log general type surfaces is used to show that
every such tower does indeed come from a family of Del Pezzo surfaces of degree 4.

In Section \ref{spectralcover},
we recall the construction of the spectral cover;
in particular, the spectral curve in \eqref{eqn.toweroverP1} comes embedded in a
Hirzebruch surface $\F\to \bP^1$.
It is then relevant to understand the monodromy of 2-torsion of the Jacobian of curves in
Hirzebruch surfaces, which is described in Section \ref{curvesmonodromy}.
Section \ref{sec.nonsingline} contains a classical treatment of
degree 4 Del Pezzo surfaces containing a line disjoint from the
singular locus, an ingredient in the comparison of moduli problems.
We discuss the connection between moduli of Del Pezzo surfaces of degree 4 and
of binary quintic forms in Section \ref{quintic}.
In Section \ref{modulistacks} we begin the study of moduli of degree 4
Del Pezzo surfaces, and we introduce genericity conditions on families
of degree 4 Del Pezzo surfaces in Section \ref{sect.genericity}.
Section \ref{loggeneraltype} recalls the moduli problem of log general
type surfaces, which is related to our moduli problem in Section \ref{comparison}.
In Section \ref{maps} we state and prove our main theorems, describing and
enumerating the components of general families of degree 4 Del Pezzo surfaces
over $\bP^1$.
In an Appendix, we show that the discrete invariant of families introduced
here agrees with the height defined in \cite{HT13}.

\medskip
\noindent\textbf{Acknowledgements.}
The first author was supported by NSF grants 0901645, 0968349, and 1148609;
The second author was supported by the SNF.
The third author was supported by NSF grants 0968349 and 1160859.

\section{Spectral cover}
\label{spectralcover}

We work over a base field $k$, perfect of characteristic different from $2$.
A Del Pezzo surface is a smooth projective surface (always assumed
geometrically integral) $X$ with ample anticanonical line bundle $\omega_X$.
The degree of $X$ is the self-intersection of the anticanonical class.
A Del Pezzo surface of degree 4 is embedded by the
anticanonical linear system as a complete intersection of two quadrics
in $\bP^4$, and is geometrically isomorphic
to the blow-up of $\bP^2$ at 5 points in general position
(i.e., no three on a line, and not all lying on a conic).
Geometrically, the curves with self-intersection $-1$ are the
16 lines on such a surface, the Picard group has
rank 6, and the primitive Picard group (i.e., the subgroup orthogonal to the
anticanonical class) is a root lattice of type $\sD_5$.
In particular, all Galois symmetries factor through the Weyl group $W(\sD_5)$.

Singular Del Pezzo surfaces, i.e., normal projective surfaces with
$\mathsf A\mathsf D\mathsf E$-singularities and ample anticanonical class, are extensively studied,
e.g., \cite{coraytsfasman}.
Such a surface $X$ has a minimal resolution $\widetilde{X}$.
The anticanonical linear system (or a suitable multiple) induces the morphism
$\widetilde{X}\to X$.
When $X$ is geometrically
a Del Pezzo surface of degree 4 with one $\sA_1$-singularity,
$\widetilde{X}$ has one curve with self-intersection $-2$, contracted under
the morphism $\widetilde{X}\to X\subset \bP^4$, and has, geometrically,
12 curves with self-intersection $-1$.
Unlike the smooth case, the
pencil of quadrics containing $X$ has a distinguished member, appearing
with multiplicity 2 in the degeneracy locus.

Let $T$ be a $k$-scheme of finite type.
Any flat family of possibly singular
degree 4 Del Pezzo surfaces $\pi\colon \cX\to T$
gives rise to a degree 5 cover $D\to T$ which encapsulates the
degeneracy loci of the pencils of quadrics associated with the
fibers of $\pi$, as follows.
The relative anticanonical line bundle $\omega^{-1}_{\pi}$ is ample and 
induces a closed immersion
$$
\cX\to \bP((\pi_*\omega_\pi^{-1})^\vee)
$$
over $T$;
we let $\pi$ also denote
projection $\bP((\pi_*\omega_\pi^{-1})^\vee) \ra \bP^1$.
The composition 
\begin{align} \label{eqn:composition}
(\wedge^5(\pi_*\omega_\pi^{-1})^\vee)^{{\otimes}2}
&\stackrel{\det}\to
\Sym^5 (\Sym^2 ((\pi_*\omega_\pi^{-1})^\vee)) \notag \\
&\cong
\Sym^5((\Sym^2 \pi_*\omega_\pi^{-1})^\vee)
\to
\Sym^5((\pi_*\cI_{\cX}(2))^\vee)
\end{align}
gives rise to an ideal sheaf on
$\bP(\pi_*(\cI_{\cX}(2)))$ and hence the \emph{spectral cover}
\[
D\subset \bP(\pi_*(\cI_{\cX}(2)))\to T.
\]

\section{Curves and their monodromy}
\label{curvesmonodromy}
In this section, we study monodromy groups of some families of curves
in Hirzebruch surfaces.
We recall that Hirzebruch surfaces, being smooth complete toric varieties,
have the property \cite[Thm.\ 6.1.15]{CLS} that every ample line bundle is very ample.

\begin{lemm}
\label{lem.monodromy}
Let $D$ be a smooth curve on $\bP^1\times \bP^1$ of bidegree $(a,b)$, with
$a\ge b\ge 3$, respectively a smooth curve on the Hirzebruch surface $\F_1$
in the class $af+b\xi$, with $a>b\ge 3$, where $f$ denotes the class of
a fiber of $\F_1\to\bP^1$ and $\xi$ denotes the class of the
$(-1)$-curve.
The monodromy action on $H^1(D,\Z)$
of the space of smooth curves in the same curve class as $D$
is the full symplectic group $\Sp(H^1(D,\Z))$,
in each of the following cases:
\begin{itemize}
\item[(i)] $D\subset \bP^1\times \bP^1$ with $a$ or $b$ odd;
\item[(ii)] $D\subset \F_1$ with $a$ even or $b$ odd.
\end{itemize}
\end{lemm}

\begin{proof}
In each case the class of $D$ is very ample.
Hence the discriminant hypersurface in $|D|$ is irreducible with
generic point corresponding to a curve with a single node;
see, e.g., \cite[\S 1]{lamotke}.
The method outlined in \cite{beauville} is applicable, provided that we
verify:
\begin{itemize}
\item
$D$ may be degenerated to acquire an $\mathsf E_6$-singularity, such that the linear series
$|D|$ is versal for this singularity;
\item
$D$ degenerates to a union $D'\cup D''$ of two smooth curves meeting
transversally in an odd number of points. 
\end{itemize}
We recall the basic strategy.
There exists a Lefschetz pencil of curves in $|D|$.
The monodromy of the pencil is generated by symplectic reflections by vanishing cycles associated with each nodal fiber.
Lefschetz theorems \cite[p.~151]{GM} imply the full monodromy group is also generated by these reflections.
Moreover, the curves in $|D|$ with precisely one node form a connected set, so the
reflections are all conjugate under the monodromy group.  Squares of symplectic reflections generate 
the kernel of 
$$
\Sp(H^1(D,\bZ))\ra \Sp(H^1(D,\bZ/2\bZ)),
$$ 
so we only have to address the representation
$\pmod 2$.  The first assumption implies that the monodromy contains $W({\mathsf E}_6)$, as it contains
all the reflections associated with simple root/vanishing cycles; it
is thus either $\Sp(H^1(D,\bZ))$
or surjects onto a subgroup $O(q) \subset \Sp(H^1(D,\bZ/2\bZ))$ 
preserving a quadratic form $q$ with $q(\delta)=1$ for each vanishing cycle $\delta$
\cite[Th.~3]{beauville}.
The monodromy can factor through such a subgroup, e.g., when $a$ and $b$ are
both even, in case (i).
The second assumption precludes this:  Given a smoothing of $D' \cup D''$ in $|D|$, fix vanishing
cycles $\delta_1,\ldots,\delta_{2p+1}$ indexed by the nodes; the sum $\sum_{j=1}^{2p+1}\delta_j$ is homologous
to zero.  Since $q(\sum_j \delta_j)=0$ we must have $q(\delta_j)=0$ for some $j$, a contradiction.

To verify the ${\mathsf E}_6$ condition, we use the plane quartic
$C=\{y^3=x^4\}$, which has a singularity of this type, and whose
versal deformations can be realized with plane quartics.  
The image of $C$ under the linear system of quadrics through two general points of $C$ gives a nodal
curve in $\bP^1 \times \bP^1$ of bidegree $(3,3)$ with the same singularity.  
Adding appropriate fibers gives
curves with all desired bidegrees.  Similarly, blowing up a generic point of $C$ gives a curve
in $\F_1$ with the desired singularity and class $4f+3\xi$.  Adding lines (with class $\xi+f$)
and fibers gives the classes we seek.  
The final condition can be checked case by
case, e.g., for $\F_1$ with $a$ even consider $[D']=f+\xi$ and
$[D'']=(a-1)f+(b-1)\xi$, so $[D']\cdot [D'']=a-1$.
\end{proof}

\begin{exam}
\label{examhyperell}
Consider hyperelliptic curves of genus $g>1$.  Note that the general such curve---with 
the datum of a line bundle of degree $g+1$---arises as
a curve of bidegree $(g+1,2)$ in $\bP^1 \times \bP^1$.  
The monodromy of such curves has been studied \cite{ACampo}:  It is an
explicit subgroup $\Gamma \subset \Sp(H^1(D,\bZ))=\Sp_{2g}(\bZ)$, where
$$\Gamma=\{\gamma \in \Sp_{2g}(\bZ): \gamma(\text{mod } 2) \in \fS_{2g+2} 
\subset \Sp_{2g}(\bZ/2\bZ)\}.$$
The symmetric group comes from the monodromy action on the branch points 
$r_1,\ldots,r_{2g+2}$ of the degree-two map $D\ra \bP^1$.  Indeed,
any two-torsion point $\eta \in \JJ(D)[2]$ admits a unique expression 
$$\eta=\sum_{j\in S} r_j - n g^1_2, \quad S\subset \{1,\ldots,2g+2\},|S|=2n, 0\le n \le g/2.$$
The monodromy representation on $\JJ(D)[2]$ for $D$ hyperelliptic therefore
factors through the permutation representation on even subsets of the branch points.
The orbits of $\JJ(D)[2]$
correspond to integers $n=0,1,\ldots,\lfloor g/2\rfloor$.
\end{exam}

\section{Nonsingular lines on degree 4 Del Pezzo surfaces}
\label{sec.nonsingline}

We summarize the results of this section:
Assume the base field $k$ is perfect with characteristic different from $2$.
Let $X\subset \bP^4$ be a
complete intersection of two quadrics which is normal
and contains a line $L$
disjoint from the singular locus of $X$. We will call such $L$ a
\emph{nonsingular line}.
Projection from $L$ identifies $X$ with the
blow-up of the projective plane along a degree 5 subscheme $\Xi$
of a smooth conic $B$, making
$X$ a degree 4 Del Pezzo surface with
restricted singularities.
The conic, which is the image of $L$ under the projection, is
canonically identified with the pencil of
quadric hypersurfaces containing $X$, so that the
the locus of
singular members of the pencil (which carries a natural scheme structure)
corresponds to $\Xi$.

Let us write $\bP^4=\bP(k^5)$ and
$L=\bP(V)$ with $V\subset k^5$ a subspace of dimension 2.

\begin{prop}
\label{prop.linetgtplanes}
With the above notation, the morphism
$L\to \bP(k^5/V)$, sending $p\in L$ to the tangent plane $T_pX$,
is an isomorphism onto a conic $B\subset \bP(k^5/V)$.
\end{prop}

\begin{proof}
Under the morphism $L\to \bP(k^5/V)\cong \bP^2$, the tautological rank two
quotient bundle $\cO_{\bP^2}^3\to \cQ$ pulls back to
$\cO_L^3\to \cO_L(1)^2$ given by a $2\times 3$ matrix of linear forms.
The morphism is therefore a closed immersion of degree 2.
\end{proof}

Projection from $L$ is a morphism
\[ \psi\colon X\to \bP(k^5/V)\cong \bP^2 \]
sending $p\in X\setminus L$ to the linear span of $L$ and $p$, and sending
$p\in L$ to $T_pX$.
A general hyperplane in $\bP^4$ containing $L$ intersects $X$ in the
union of $L$ and a
residual cubic curve having intersection number 2 with $L$.
The residual cubic curve belongs to the linear system
$\bP(H^0(X,\psi^*\cO_{\bP^2}(1)))$.
Its self-intersection number is 1, hence
the morphism $\psi$ is \emph{birational}.

Under $\psi$, a general member $D$ of the linear system
$\bP(H^0(X,\cO_X(1)))$ maps to an irreducible cubic curve $C\subset \bP^2$, and
the image linear system is spanned by:
\begin{itemize}
\item $B\cup \ell$, $B\cup \ell'$, $B\cup \ell''$ where
$\ell$, $\ell'$, $\ell''$ span
$\bP(H^0(\bP^2,\cO_{\bP^2}(1)))$;
\item an irreducible cubic curve $C$ as above;
\item another such irreducible cubic curve $C'$, with $B\cap C'\ne B\cap C$.
\end{itemize}
Let $\Xi=B\cap C\cap C'$ be the base locus of the linear system;
comparing self-intersection numbers of $D$ and $C$ we see that
$\deg(\Xi)=5$, and the linear system determines a morphism
\begin{equation}
\label{eqn.rho}
\rho\colon \Bl_{\Xi}(\bP^2)\to X.
\end{equation}
It is known that $\Bl_{\Xi}(\bP^2)$ is normal, and the fiber over
a geometric point of $\Xi$
is a copy of $\bP^1$ mapping to a line in $X$.
It follows that
$\rho$ is birational and finite, hence by
Zariski's main theorem is an isomorphism.

Let $Q$ be a quadric hypersurface containing $X$.
To $Q$ there is an associated symmetric bilinear form on $k^5$ (defined up
to scalar multiplication).
Since $L$ avoids the singular locus of $Q$,
this is a bilinear form of rank $\ge 4$, and $\bP(V^\perp)$ is
a plane containing $L$; we therefore have a morphism
\[
\varphi\colon \bP(H^0(\bP^4,\mathcal{I}_X(2)))\to \bP(k^5/V).
\]

\begin{lemm}
\label{lem.planecontL}
For any plane
$\Pi\subset \bP^4$ containing $L$ exactly one of the following statements is
true:
\begin{itemize}
\item[(i)]
$\Pi\notin B$, and as schemes we have
$\Pi\cap X=L\cup \{p\}$ for a point $p\in X$ not
in $L$ nor in any line $L'\subset X$ satisfying $L\cap L'\ne \emptyset$;
\item[(ii)]
We have $\Pi\in B$, i.e., $\Pi=T_pX$ for a unique $p\in L$, and
the scheme $\Pi\cap X$
is irreducible, has $L$ as reduced subscheme, and has a
unique embedded point, located at $p$;
\item[(iii)] $\Pi\in B$, and $\Pi\cap X=L\cup L'$ where $L'\subset X$ is a line
satisfying $L\cap L'=\{p\}$, with $\Pi=T_pX$.
\end{itemize}
\end{lemm}

\begin{proof}
If the pencil of quadric hypersurfaces containing $X$ restricts to a pencil
of conics in $\Pi$, then it decomposes as $L$ plus a residual pencil of lines
with a base point $p$.
If $p\notin L$ then we are in case (i), with $\Pi$ the linear span of
$L$ and $p$ and no line on $X$ through $p$ meeting $L'$.
If $p\in L$ then we are in case (ii).
Otherwise some member of the pencil of quadric hypersurfaces
contains $\Pi$, and for any other member $Q$ of the pencil we have
$\Pi\cap X=\Pi\cap Q$ of degree 2 in $\Pi$ and containing $L$.
We exclude $L$ having multiplicity 2 by Proposition \ref{prop.linetgtplanes},
and obtain therefore $L\cup L'$ as in case (iii).
\end{proof}

By Lemma \ref{lem.planecontL},
the morphism $\varphi$ factors through $B$.
Since by a Chern class computation the morphism $\varphi$ has degree $2$,
the morphism $\varphi$ determines an \emph{isomorphism}
\begin{equation}
\label{eqn.isopenciltoL}
\bP(H^0(\bP^4,\mathcal{I}_X(2)))\cong L.
\end{equation}
The morphism $\varphi$ sends singular members of the pencil of quadric
hypersurfaces to $\Xi$ and nonsingular members to
$B\smallsetminus \Xi$.

\begin{theo}
\label{thm.DP4asblowup}
Let $k$ be a perfect field of characteristic different from $2$,
and let $X\subset \bP^4=\bP(k^5)$ be a Del Pezzo surface of degree 4
containing a nonsingular line $L=\bP(V)$;
projection from $L$ identifies $X\cong \Bl_\Xi(\bP(k^5/V))$
with $\Xi\subset B=\{T_pX|p\in L\}$ uniquely determined of degree 5.
Then the isomorphism $\bP(H^0(\bP^4,\mathcal{I}_X(2)))\cong B$,
sending a quadric hypersurface $Q\supset X$
to the plane $\bP(V^\perp)$ with $V^\perp$ the orthogonal space to $V$ under
a symmetric bilinear form corresponding to $Q$, identifies the
scheme of singular members of the pencil of quadric hypersurfaces containing $X$
(defined by the vanishing of a determinant in an evident fashion) with $\Xi$.
\end{theo}

When $X$ is smooth
the isomorphism \eqref{eqn.isopenciltoL} and statement of
Theorem \ref{thm.DP4asblowup} are classical; see \cite{skorobogatov}.
For singular $X$,
projection from nonsingular lines appears as a key ingredient already
in Segre's classical treatment \cite{segre}.

\begin{proof}[Proof of Theorem \ref{thm.DP4asblowup}]
Let $U$ be an irreducible smooth affine variety with $k$-point $u\in U$,
$\pi\colon \cX\to U$ a generically smooth family of Del Pezzo surfaces of degree 4,
and $\cL\subset \cX$ a family of lines; fix an identification
$\cX_{u}=\pi^{-1}(u)\cong X$ sending $\cL_u$ to $L$.
There are versions for families of the identifications
$\Bl_{\Xi}(\bP(k^5/V))\cong X$ and 
$\bP(H^0(\bP^4,\cI_X(2)))\cong L$
of \eqref{eqn.rho} and \eqref{eqn.isopenciltoL} respectively.
So we have subschemes of $\cL$ finite and flat over $U$,
corresponding to the spectral cover and the center of the blow-up, respectively.
Since they agree over the generic point of $U$, they must be equal.
\end{proof}

\section{Invariants of binary quintics}
\label{quintic}

We consider the space $\bP^5$ of binary quintic forms 
$\sum_{i=0}^5 A_ix^{5-i}y^i$
with standard $\PGL_2$-action.
Let $U\subset \bP^5$ denote the binary quintic forms with at most
double roots.  
Over any field,
this is the semistable and stable locus.

The algebra of invariant homogeneous polynomials is generated by 
classically known invariants
$I_d$ of degrees $d=4$, $8$, $12$, and $18$,
given explicitly in \cite[pp.~87-89]{schur} with coefficients in $\bQ$.
Moreover, $I^2_{18}$ may be expressed as a weighted-homogeneous
form $F\in \bQ[I_4,I_{8},I_{12}]$ of degree $36$
(see, e.g., \cite[\S~7.2]{gulbrandsen}).
So the invariant-theoretic quotient is
\[\Proj(\bQ[I_4,I_8,I_{12}]) = \bP(1,2,3)_{\bQ}.\]

The $\PGL_2$-action on $U$
has reduced finite stabilizer group schemes.
By \cite[(8.1)]{LMB} the stack quotient $[U/\PGL_2]$ is a
Deligne-Mumford stack, and is in fact
a separated Deligne-Mumford stack
(the standard valuative criterion for separation
may be checked easily).
It follows that the
evident morphism $[\overline{M}_{0,5}/\fS_5]\to [U/\PGL_2]$
contracting any component of a stable $5$-pointed genus $0$ curve
with exactly two marked points is proper; it is as well surjective and
is an isomorphism on the locus of five distinct points
on $\bP^1$.
In particular, $[U/\PGL_2]$ is proper over $\Spec(\Z)$.

Choose $J_4,J_8,J_{12} \in \bZ[A_0,\ldots,A_5]$ such that
$$\bZ[J_4,J_8,J_{12}]_d=\bZ[A_0,\ldots,A_5]_d \cap \bQ[I_4,I_8,I_{12}], \quad \text{for}\ d=4,8,12,$$
i.e., $J_4,J_8,$ and $J_{12}$ generate the invariants
in degrees $4,8,$ and $12$ over $\bZ$.  
Let $J_{18}$ be a multiple of $I_{18}$ with relatively prime integer coefficients.
A direct computation via Gr\"obner bases
shows the invariants $J_4$, $J_8$, and $J_{12}$ define a morphism
$[U/\PGL_2]\to \bP(1,2,3)$ {\em over the integers.}
It follows from $\Pic(\overline{M}_{0,5})^{\fS_5}\cong\Z$ that the
composite
\[ [\overline{M}_{0,5}/\fS_5]\to [U/\PGL_2]\to \bP(1,2,3)\]
is quasi-finite, hence each individual morphism is proper and quasi-finite.

Since, over a field of characteristic zero,
$\bP(1,2,3)$ is the coarse moduli space of $[U/\PGL_2]$, and also of
$[\overline{M}_{0,5}/\fS_5]$,
the same is then true over $\Spec(\Z)$ by an application of Zariski's Main Theorem
to the normal scheme $\bP(1,2,3)$, and over a field of positive characteristic
by the same argument, noting that the scheme loci of
$[\overline{M}_{0,5}/\fS_5]$ and of $[U/\PGL_2]$ are dense over any field.

\begin{lemm}
\label{UmodPGL2}
The stacks $[ \bP^5 / \PGL_2 ]$ and $[ U / \PGL_2 ]$
have Picard group isomorphic to $\bZ$,
generated by $H:=[Z(J_{18})]-4[Z(J_4)]$.
\end{lemm}

\begin{proof}
We may identify $[ \bP^5 / \PGL_2 ]$ with
$[ (\bA^6 \smallsetminus \{0\}) / \GL_2 ]$ with standard $\GL_2$-action
is twisted by the $(-2)$-power of the determinant representation.
Since the open immersions of regular stacks
\[ [ U / \PGL_2 ] \subset [ \bP^5 / \PGL_2 ]
\cong [ (\bA^6 \smallsetminus \{0\}) / \GL_2 ] \subset [ \bA^6/\GL_2 ] \]
each have complement of codimension $\ge 2$, they induce isomorphisms of
Picard groups, hence the Picard groups are identified with
\[ \Pic([\bA^6/\GL_2])\cong \Pic(B\GL_2)\cong \bZ. \]
(For these identifications of Picard groups, see \cite[Lem.\ 2]{edidingraham}.)
By comparing characters of $\GL_2$ we see that $H$ generates the
Picard group.
\end{proof}

For the rest of this section
we work over a field $k$ of characteristic different from $2$.
On the space of binary quintic forms the discriminant is an invariant of
degree $8$, defining a divisor $\Delta\subset \bP^5$ which on $U$
has singularities along quintic forms with two double roots.

\begin{lemm} \label{lemm:Kclass}
The canonical class $K_{[U/\PGL_2]}$ is $-3H$. 
\end{lemm}
\begin{proof} 
The coarse moduli space $\bP(1,2,3)$ has a standard affine chart
isomorphic to $\bA^2$ and on it a standard generator of the
canonical bundle on this chart.
This pulls back to a rational section of the canonical bundle of
$[U/\PGL_2]$ vanishing to order 1 along $Z(J_{18})$ and having a pole
of order 6 along $Z(J_4)$.
So we have
\[ K_{[U/\PGL_2]}=[Z(J_{18})]-6[Z(J_4)]=9H-12H=-3H. \]
\end{proof}

\begin{lemm}
\label{M05modS5}
We have
\[
\Pic([ \overline{M}_{0,5} / \fS_5 ])\cong \bZ\oplus \bZ/2\bZ,
\]
where the first summand is the Picard group of $[ U / \PGL_2 ]$ and
the second summand is generated by
$[\partial]-2H$, where
$\partial=[\partial \overline{M}_{0,5} / \fS_5]$ with
$\partial \overline{M}_{0,5}$ denoting the boundary
(complement of $M_{0,5}$) of $\overline{M}_{0,5}$.
\end{lemm}

\begin{proof}
As in the proof of Lemma \ref{UmodPGL2} we may remove any codimension $2$
substack without changing the Picard group.
Removing the locus of curves with three irreducible components,
respectively quintic forms with two double roots, the morphism
$[\overline{M}_{0,5}/\fS_5]\to [U/\PGL_2]$ restricts to a morphism
\begin{equation}
\label{MprimeUprime}
[ \overline{M}_{0,5}' / \fS_5 ] \to [ U' / \PGL_2 ],
\end{equation}
such that $\Delta\cap U'$ is a smooth divisor on $U'$.
A construction called \emph{root stack} adds stabilizer along a divisor:
with standard $\bG_m$-action on $\bA^1$ the stack
$[\bA^1/\bG_m]$ is identified with pairs consisting of a line bundle and
a global section, which is determined by an effective Cartier divisor
so we have a morphism $[ U' / \PGL_2 ] \to [\bA^1/\bG_m]$;
then the root stack of interest is
\[ [ U' / \PGL_2 ]\times_{[\bA^1/\bG_m],\theta_2}
[\bA^1/\bG_m] \]
where
\[
\theta_2\colon [\bA^1/\bG_m]\to [\bA^1/\bG_m]
\]
is the morphism induced by squaring on both $\bA^1$ and $\bG_m$;
cf.\ \cite[\S 2]{cadman} and \cite[App.\ B]{AGV}.
Since $\Delta$ acquires multiplicity $2$ upon pullback to
$\overline{M}_{0,5}$, the morphism \eqref{MprimeUprime} factors through
the root stack; the morphism to the root stack is quasi-finite, proper,
and representable, hence by Zariski's Main Theorem is an isomorphism.
The result now follows by the discussion of the
Picard group of a root stack in \cite[\S3.1]{cadman} and the fact that
the $\PGL_2$-equivariant class of $\Delta$ is
$4H\in \Pic([ U / \PGL_2 ])$.
\end{proof}

\section{Moduli stacks of degree 4 Del Pezzo surfaces}
\label{modulistacks}
The moduli stack $\mathcal{M}^{\mathrm{DP4}}$,
of degree 4 Del Pezzo surfaces, where
we allow arbitrary $\mathsf{ADE}$-singularities,
is an Artin stack, smooth, of finite type, with separated quasicompact diagonal, and of
relative dimension 2 over any chosen base ring.

Given a flat family of degree 4 Del Pezzo surfaces $\cX\to S$ over an arbitrary
scheme $S$, standard Hilbert scheme machinery gives rise to a scheme of
nonsingular lines $\cL\to S$.
By the associated deformation theory,
specifically the result \cite[Cor.\ 5.4]{TDIV} and the fact that
nonsingular line $L\subset X$ satisfies
$H^i(L,N_{L/X})=0$ for $i=0$ and $1$,
the morphism $\cL\to S$ is \'etale.
In particular, containing a nonsingular line is an open condition in moduli,
and there is a corresponding moduli stack $\mathcal{M}^\circ$
of degree 4 Del Pezzo surfaces containing a nonsingular line.
This comes with a representable \'etale covering
\begin{equation}
\label{etalecov}
\mathcal{M}'^\circ\to \mathcal{M}^\circ,
\end{equation}
where $\mathcal{M}'^\circ$ is the moduli stack of
degree 4 Del Pezzo surface with choice of nonsingular line.
The stacks $\mathcal{M}^\circ$ and $\mathcal{M}'^\circ$ are also smooth,
of finite type, and of relative dimension 2 over the base ring.

We work over a perfect field $k$ with $\chara k\ne 2$.
By \cite[\S 2.2]{skorobogatov}, a smooth Del Pezzo surface
$X\subset \bP^4$ over $k$
is specified uniquely up to isomorphism
by the singular locus of the pencil of quadric hypersurfaces
$D\subset \bP(H^0(X,\mathcal{I}_X(2)))$ up to projective equivalence
together with an isomorphism class of
$k$-torsors under the group scheme $R_{k[D]/k}(\mu_2)/\mu_2$,
where $R_{k[D]/k}$ denotes restriction of scalars.
This description extends naturally to
singular degree 4 Del Pezzo surfaces containing a nonsingular line
(Theorem \ref{thm.stackthm} and Corollary \ref{cor.simplytransitive}).

\begin{theo}
\label{thm.stackthm}
The blow-up of the projective plane along the zero locus of a
binary quintic form on the Veronese-embedded projective line yields an
isomorphism of stacks
\begin{equation}
\label{stackiso}
[ \bP^5 / \PGL_2 ]\eqto \mathcal{M}'^\circ
\end{equation}
with inverse isomorphism given by the spectral cover construction
(Section \ref{spectralcover}).
The spectral cover morphism
\begin{equation}
\label{eqn.MtoP5modPGL2}
\mathcal{M}^\circ \to [ \bP^5 / \PGL_2 ]
\end{equation}
is an \'etale gerbe (an \'etale surjective morphism with \'etale surjective
relative diagonal) which is neutral, i.e., admits a section; a section is
the composite of \eqref{etalecov} and \eqref{stackiso}.
\end{theo}

\begin{proof}
That we have the morphism \eqref{stackiso} is clear.
As remarked in the proof of Theorem \ref{thm.DP4asblowup},
the treatment given in Section \ref{sec.nonsingline} can be carried out
in a relative setting over a smooth $k$-scheme, so by
Theorem \ref{thm.DP4asblowup} in a relative setting we have
the isomorphism as claimed.
The claim about the section to the morphism \eqref{eqn.MtoP5modPGL2} is clear,
and implies that the morphism \eqref{eqn.MtoP5modPGL2} is \'etale surjective.
It remains, therefore, only to show that the relative diagonal is
\'etale surjective, i.e., that two families in $\mathcal{M}^\circ$ having
a common spectral cover are locally isomorphic.
We can \'etale locally make choices of nonsingular lines in the fibers, and
then we apply the isomorphism \eqref{stackiso} to establish the assertion.
\end{proof}

\begin{coro}
\label{cor.simplytransitive}
Given a Del Pezzo surface $X\subset \bP^4$ of degree $4$ over $k$ the set of nonsingular lines
defined over $k$ is either empty or is acted upon simply transitively
by the kernel of $\Aut(X)\to \PGL(H^0(X,{\mathcal I}_X(2)))$.
\end{coro}

In the stack of binary quintic forms there is the open substack
$[ U / \PGL_2 ]$ treated in Section \ref{quintic}, consisting of
binary quintic forms with at most double roots.
We let $\mathcal{M}$ denote the corresponding open substack of
$\mathcal{M}^\circ$, under the morphism \eqref{eqn.MtoP5modPGL2}.
Concretely, $\mathcal{M}$ is the moduli stack of degree 4 Del Pezzo surfaces
with at most one $\sA_1$-singularity, or with two $\sA_1$-singularities connected
by a line.
The morphism \eqref{eqn.MtoP5modPGL2} restricts to
\begin{equation}
\label{eqn.MtoUmodPGL2}
\mathcal{M}\to [ U / \PGL_2 ],
\end{equation}
also a neutral \'etale gerbe.
However, even though $[ U / \PGL_2 ]$ is separated (and, in fact, is
proper), the stack $\mathcal{M}$ is nonseparated, since
as remarked in the
Introduction (see also \cite[Rem.\ 2]{HT13}),
the order of the geometric stabilizer group in a family may
decrease under specialization.

If we let $U^\circ$ denote the open subset of $U$ with
nonvanishing discriminant and $\mathcal{M}^\mathrm{sm}$ denote the
moduli stack of smooth degree 4 Del Pezzo surfaces, then
\eqref{eqn.MtoUmodPGL2} restricts to a neutral \'etale gerbe
\begin{equation}
\label{eqn.MsmtoUcircmodPGL2}
\mathcal{M}^{\mathrm{sm}}\to [ U^\circ / \PGL_2 ].
\end{equation}

\begin{rema}
\label{badcoarsemodulispace}
The morphism $\mathcal{M}\to \bP(1,2,3)$ is
one-to-one on geometric points by \cite[Prop.\ 1]{HT13}.
By combining a standard property of gerbes \cite[Lem.\ 3.8]{LMB} with
the fact that $[U/\PGL_2]\to \bP(1,2,3)$ is a coarse moduli space,
we deduce that the morphism $\mathcal{M}\to \bP(1,2,3)$ is also universal
for morphisms to algebraic spaces.
However, $\mathcal{M}$ is nonseparated,
the property of being \'etale locally on a coarse moduli space a
quotient of a scheme by a finite group
(a standard property for separated Deligne-Mumford stacks and more generally
for Deligne-Mumford stacks with finite stabilizer \cite{keelmori}) fails to
hold for the stack $\mathcal{M}$.
\end{rema}

\begin{coro}
\label{gerbecor}
The singular locus of the total space of the universal family over
\[
[ \overline{M}_{0,5} / \fS_5 ] \times _{ [ U / \PGL_2 ]}
\mathcal{M}
\]
consists of ordinary double points in the fibers over
the locus of singular degree 4 Del Pezzo surfaces.
\end{coro}

\begin{prop}
\label{prop:picardM}
The morphism \eqref{eqn.MtoP5modPGL2} and \eqref{eqn.MtoUmodPGL2}
induce isomorphisms on Picard groups.
In particular, we have $\Pic(\mathcal{M}^{\mathrm{DP4}})\cong
\Pic(\mathcal M)\cong \Z$.
\end{prop}

\begin{proof}
Let
\[ \beta\colon
[\bP^5/\PGL_2]\to \mathcal{M}^\circ
\]
denote the composite of \eqref{etalecov} and \eqref{stackiso},
mentioned in the statement of Theorem \ref{thm.stackthm}.
The Leray spectral sequence gives an exact sequence
\[0\to \Pic([\bP^5/\PGL_2])\to \Pic(\mathcal{M}^\circ)\to
\Hom(\Aut(\beta),\bG_m)\to 0.\]
Already the restriction of the sheaf
$\mathcal{H}om(\Aut(s),\bG_m)$ to $[U^\circ/\PGL_2]$ has no nontrivial sections.
Since the stacks in question all have separated diagonal, this is enough
to deduce the vanishing of $\Hom(\Aut(\beta),\bG_m)$.
We conclude by appealing to Lemma \ref{UmodPGL2}.
\end{proof}

We call attention to the restriction of $\beta$ to $[U/\PGL_2]$,
a section
\begin{equation}
\label{eqn.scnofMtoUmodPGL2}
[U/\PGL_2]\to \mathcal{M}
\end{equation}
of the gerbe \eqref{eqn.MtoUmodPGL2}, and to $[U^\circ/PGL_2]$, a section
\begin{equation}
\label{eqn.scnofMsmtoUcircmodPGL2}
[U^\circ/\PGL_2]\to \mathcal{M}^\mathrm{sm}
\end{equation}
of the gerbe \eqref{eqn.MsmtoUcircmodPGL2}.

\begin{coro}
\label{coro:numbericalinvariant}
For a flat family $\pi\colon \cX\to \bP^1$ of degree 4 Del Pezzo surfaces with 
$\mathsf{ADE}$-singularities, we have
\begin{equation}
\label{degequalsdeg}
\deg(\pi_*\omega_\pi^{-1})=\deg(\pi_*(\cI_{\cX}(2))),
\end{equation}
with the notation of Section \ref{spectralcover}.
\end{coro}

\begin{proof}
By Proposition \ref{prop:picardM}, the degrees in
\eqref{degequalsdeg} must be related by a
constant proportionality.
We deduce their equality from
any of the worked out examples,
e.g., Case 1 on page 11 of \cite{HT13} with
$\pi_*\omega_\pi^{-1}\cong \cO_{\bP^1}(-2n)^5$
and $\pi_*(\cI_{\cX}(2))\cong \cO_{\bP^1}(-5n)^2$.
\end{proof}

\begin{defi}
\label{height}
The height of a
flat family $\pi\colon \cX\to \bP^1$ of degree 4 Del Pezzo surfaces with 
$\mathsf{ADE}$-singularities is the quantity
\[
h(\cX)=-2\deg(\pi_*\omega_\pi^{-1})=-2\deg(\pi_*(\cI_{\cX}(2))).
\]
\end{defi}

The constant 2 in the definition of height is a convention.
This height agrees with the height defined in \cite[\S 3]{HT13}; see
Appendix.

\begin{lemm}
\label{lem:hirzebruchsurface}
Let $\F_n$ be the Hirzebruch surface, with Picard group
generated by $(-n)$-curve $\xi$ and fiber $f$.
If $D\subset \F_n$ is a reduced divisor in
class $[D]=af+b\xi$ then $(b-1)n \le a$.
If $D\subset \F_n$ is irreducible then either $a\ge bn$ or $a=0$.
\end{lemm}

\begin{proof}
The first assertion is a consequence of Proposition 2.2 of \cite{laface}.
The second encodes the fact that 
$D\cdot \xi\ge 0$ unless $D$ contains a multiple
of $\xi$.  
\end{proof}

For $d\le e$, on the Hirzebruch surface
$\bP(\cO_{\bP^1}(d)\oplus \cO_{\bP^1}(e))\cong \F_n$ with $n=e-d$ we have
$c_1(\cO_{\bP(\cO(d)\oplus \cO(e))}(1))=-df+\xi$.

\begin{lemm}
\label{lem:easyintersectiontheory}
Let $E$ be a vector bundle on a scheme $B$,
let $\pi:\bP(E^\vee)\to B$ be the projectivization of the dual of $E$,
let $F$ be a vector bundle of rank $f$ on $B$, and let
$F\to \Sym^d(E)$ for some $d\ge 1$ be given, defining
$\cX\subset \bP(E^\vee)$
(locally by $f$ homogeneous equations of degree $d$).
Assume that the fibers of $\cX\to B$ are of codimension $f$.
Then
\[
[\cX]=c_f(\cO_{\bP(E^\vee)}(d)\otimes \pi^*F^\vee)
\]
in the Chow group of $\bP(E^\vee)$.
\end{lemm}

\begin{proof}
From $F\to \Sym^d(E)=\pi_*\cO_{\bP(E^\vee)}(d)$ we get
$\pi^*F\to \cO_{\bP(E^\vee)}(d)$ and hence a global section of
$\cO_{\bP(E^\vee)}(d)\otimes \pi^*F^\vee$,
whose vanishing defines $\cX\subset \bP(E^\vee)$.
\end{proof}

We will consider flat families of degree 4 Del Pezzo surfaces over $\bP^1$ with
smooth general fiber.
The geometric fibers over closed points of $\bP^1$ are allowed to have
arbitrary $\mathsf{ADE}$-singularities.
Let $h=h(\cX)$ be the height of such a family
$\pi\colon \cX\to \bP^1$.

\begin{prop}
\label{discriminantdegree}
A generically smooth family $\pi\colon \cX\to \bP^1$
of degree 4 Del Pezzo surfaces
of height $h=h(\cX)$ has discriminant divisor $\Delta(\pi)\subset \bP^1$
of degree $2h$.
\end{prop}

\begin{proof}
The discriminant divisor is defined by the vanishing of
\begin{align*}
\cO_{\bP^1}(-10h)&\cong
((\wedge^2\pi_*(\cI_{\cX}(2))^\vee)^\vee)^{\otimes 20} \\
&\to
\Sym^8 ( \Sym^5 (\pi_*(\cI_{\cX}(2))^\vee)^\vee)\to
(\wedge^5 \pi_*\omega_\pi^{-1})^{\otimes 16}\cong
\cO_{\bP^1}(-8h).
\end{align*}
We are using the eighth symmetric power of the dual
of the linear transformation coming from (\ref{eqn:composition}).  
\end{proof}

\begin{coro}
\label{cor.degree}
For a family $\pi\colon \cX\to \bP^1$ of height $h$ the
morphism $\bP^1\to \bP(1,2,3)$ has degree $6h$,
i.e., the image of the class $[\bP^1]$ is $6h$ times the positive generator
of the divisor class group of $\bP(1,2,3)$.
\end{coro}

\begin{prop}
\label{noheight2prop}
Let $\pi\colon \cX\to \bP^1$ be a nonconstant flat generically smooth family of degree 4 Del Pezzo surfaces
with $\mathsf{ADE}$-singularities.
Then:
\begin{itemize}
\item
We have $h(\cX)\ge 4$.
\item If the spectral cover $D$ is irreducible then
$h(\cX)\neq 6$.
\item If in addition the 
monodromy action on the lines does not factor through $\fS_5 \subset W(\sD_5)$
then $h(\cX)\neq 4$.
\end{itemize}
\end{prop}

\begin{proof}
Letting $h=h(\cX)$, it follows from the definition of height that
$(\wedge^5 \pi_*\omega_\pi^{-1})^{\otimes 2}\cong \cO_{\bP^1}(-h)$.
Lemma \ref{lem:easyintersectiontheory} yields
\[
[D]=c_1(\cO_{\bP(\pi_*(\cI_{\cX}(2)))}(5))-hf.
\]

Let us write
\[\pi_*(\cI_{\cX}(2))\cong \cO(a)\oplus \cO(-\frac{h}{2}-a)\]
with $a\le -h/4$.
We set $n=-2a-h/2$.
Then,
\begin{equation}
\label{eqnD}
[D]=(-5a-h)f+5\xi
\end{equation}
on $\bP(\pi_*(\cI_{\cX}(2)))\cong \F_n$.

Now assume that the generic fiber of $\pi$ is smooth.
Then $D$ is reduced, so by Lemma \ref{lem:hirzebruchsurface} we have
$-3a-h\le 0$.
Combining the facts, we have
\[
-\frac{h}{3}\le a\le -\frac{h}{4}.
\]
So, $h=2$ is impossible.

Suppose further that $D$ is irreducible.  
Then the second part of Lemma~\ref{lem:hirzebruchsurface}
implies 
$-5a-h \ge 5n =5(-2a-h/2),$
hence 
\[
-\frac{3h}{10}\le a\le -\frac{h}{4}.
\]
This excludes $h=6$.  

Finally, suppose $h=4$.  The analysis above implies $a=-1,n=0$ and 
$[D]=f+5\xi$, i.e., $D$ has
bidegree $(1,5)$ in $\bP^1 \times \bP^1$.  Since $D$ is irreducible, it is necessarily isomorphic
to $\bP^1$.  Since $\bP^1$ is simply connected, the monodromy action on the
families of planes in the singular quadric hypersurfaces
consists of two $\fS_5$-orbits.  
\end{proof}

\section{Genericity conditions}
\label{sect.genericity}

We continue to work over a perfect field $k$ of characteristic
different from $2$.
The primary case of interest is a generically smooth family
$\pi\colon \cX\to \bP^1$ of degree 4 Del Pezzo surfaces with
square-free discriminant.
We introduce two conditions for $\pi$ to be general.

\begin{defi}
\label{defn.generalfamily}

\

\begin{itemize}
\item[$(\mathrm{G1})$]
$\pi$ has reduced discriminant divisor,
or equivalently, $\cX$ is smooth and each fiber of $\pi$ has at worst a
single $\sA_1$-singularity,
\item[$(\mathrm{G2})$]
$\pi$ has full $W(\sD_5)$-monodromy of lines in smooth fibers.
\end{itemize}
\end{defi}

For $D\to \F\to \bP^1$ with $\F\to \bP^1$ a Hirzebruch surface with divisor
$D$ such that $D\to\bP^1$ is finite and flat of degree 5, there are two
related conditions:
\begin{defi}
\label{defn.generalfamily2}

\

\begin{itemize}
\item[$(\mathrm{G1})'$]
$D$ is simply branched over $\bP^1$.
\item[$(\mathrm{G2})'$]
The normal closure of $k(D)$ over $k(\bP^1)$ has Galois group $\fS_5$.
\end{itemize}
\end{defi}

Given a generically smooth family $\pi\colon \cX\to\bP^1$ of Del Pezzo
surfaces of degree 4, we will say that $\pi$ satisfies
$(\mathrm{G2})'$ if the spectral cover
$D\subset \bP(\pi_*(\mathcal{I}_{\cX}(2)))\to\bP^1$
satisfies $(\mathrm{G2})'$.
We remark that in $W(\sD_5)$ there are no proper subgroups strictly
containing $\fS_5$.
So, $(\mathrm{G2})'$ implies either $\fS_5$-monodromy or
full $W(\sD_5)$-monodromy of lines in smooth fibers.

\begin{prop}
\label{prop.genus}
A family $\pi\colon \cX\to \bP^1$ of height $h=h(\cX)$ satisfying $(\mathrm{G1})$ has $2h$ singular
fibers.
If, furthermore, $\pi$ satisfies $(\mathrm{G2})'$, then the spectral curve
is an irreducible nonsingular curve of genus $h-4$.
\end{prop}

\begin{proof}
This is immediate from Proposition \ref{discriminantdegree}.
\end{proof}

Condition $(\mathrm{G1})$ is open in moduli.
In the locus of moduli where $(\mathrm{G1})$ is satisfied, condition $(\mathrm{G2})$ is an
open and closed condition.
This means, if we let
$\Hom(\bP^1,\mathcal{M};h)$ denote the moduli stack of height $h$
families of Del Pezzo surfaces of degree 4,
then the families satisfying $(\mathrm{G1})$ are the points of a
well-defined open substack $\Hom_{(\mathrm{G1})}(\bP^1,\mathcal{M};h)$.
There is an open and closed substack
\[\Hom_{(\mathrm{G1}),(\mathrm{G2})}(\bP^1,\mathcal{M};h)\subset \Hom_{(\mathrm{G1})}(\bP^1,\mathcal{M};h)\]
where both $(\mathrm{G1})$ and $(\mathrm{G2})$ are satisfied.

If $D\subset \bP(\pi_*(\mathcal{I}_{\cX}(2)))\to\bP^1$ is the spectral cover
of $\pi$, then it satisfies $(\mathrm{G1})'$ if and only if
$\pi$ satisfies $(\mathrm{G1})$.
Furthermore, property $(\mathrm{G2})$ for $\pi$ implies
that $D\subset \bP(\pi_*(\mathcal{I}_{\cX}(2)))\to\bP^1$
satisfies $(\mathrm{G2})'$, although the reverse implication does not hold.
Because any subgroup of $\fS_5$ containing a transposition and acting
transitively on $\{1,\dots,5\}$ is the full group $\fS_5$,
if in Definition \ref{defn.generalfamily2} $(\mathrm{G1})'$ holds and
$D$ is irreducible, then $(\mathrm{G2})'$ holds as well.

There are open and closed substacks
\[
\Hom_{(\mathrm{G1}),(\mathrm{G2})'}(\bP^1,\mathcal{M};h)\subset \Hom_{(\mathrm{G1})}(\bP^1,\mathcal{M};h)
\]
and
\[
\Hom_{(\mathrm{G1})',(\mathrm{G2})'}(\bP^1,[U/\PGL_2];h)\subset \Hom_{(\mathrm{G1})'}(\bP^1,[U/\PGL_2];h),
\]
where the notation is self-explanatory.
We will remove $h$ from the notation when we do not want to constrain the
height.

\begin{lemm}
\label{lem.smoothexpdim}
The stack $\Hom_{(\mathrm{G1})}(\bP^1,\mathcal{M};h)$ is smooth of dimension 
$\frac{3}{2}h+2.$
\end{lemm}

\begin{proof}
For $h=0$ the assertion is trivial, so we assume $h>0$.
As remarked above,
$(\mathrm{G1})$ implies that the spectral curve $D$ is smooth.
Since $\mathcal{M}$ is \'etale over $[U/\PGL_2]$, it suffices to
show that the space of maps $\bP^1\to [U/\PGL_2]$ of degree $6h$
with smooth spectral curve is smooth.

We treat two cases, according to whether or not 
the spectral curve is irreducible.
First suppose that the spectral curve is irreducible.
The stack of morphisms $\bP^1\to B(\PGL_2)$ is
the stack $\Bun_{\PGL_2}$ whose fiber over any scheme $T$
is the category of principal $\PGL_2$-bundles over
$T\times \bP^1$.
It is a smooth algebraic stack with two irreducible components,
each of dimension $-3$.
Identifying $\PGL_2$-bundles with $\bP^1$-bundles,
these correspond to the Hirzebruch surfaces $\F_n$ with
$n$ even, respectively, $n$ odd.
The stack of morphisms from $\bP^1$ to $[U/\PGL_2]$ is an
algebraic stack
$\Hom(\bP^1,[U/\PGL_2])$, and its points correspond to covers
$D\subset \F\to \bP^1$ as above.
Standard deformation theory in a relative setting
(see \cite[Cor.\ 5.4]{TDIV}) implies that the morphism
\begin{equation}
\label{eqn.toBun}
\Hom_{(\mathrm{G1})'}(\bP^1,[U/\PGL_2];h) \to \Bun_{\PGL_2}
\end{equation}
is smooth of relative dimension $\frac{3}{2}h+5$.
Indeed, in this situation $\cO_{\F}(D)|_D$
has degree $5h/2$, so $H^1(D,\cO_{\F}(D)|_D)=0$ and 
$$
\dim H^0(D,\cO_{\F}(D)|_D)=\frac{3}{2}h+5.
$$

If the spectral curve is reducible, then it must be a
disjoint union of the exceptional curve $E\subset \F_n$ and an
irreducible curve different from $E$.
Then $h$ must be divisible by $6$, with $n=h/6$ and spectral curve
$D\cup E$ with $[D]=(2/3)hf+4\xi$, by \eqref{eqnD} (with the same notation).
By deformation theory for the curve in the surface $E\subset \F_n$,
the vanishing
\begin{equation}
\label{extvanishes}
\Ext^1(\Omega^1_{\F_n}(\log E),\mathcal{O}_{\F_n})=0
\end{equation}
implies that 
near a point with reducible spectral curve
the morphism \eqref{eqn.toBun} factors through the locally closed substack of
$\Bun_{\PGL_2}$ corresponding to the isomorphism type of $\F_n$.
We conclude as above by computing the dimension of $|D|$ to be
$\frac{5}{3}h+4$ and
noting that
$\frac{5}{3}h+4-(n+2)=\frac{3}{2}h+2$.
The vanishing \eqref{extvanishes} may be seen by comparing
the sequences of Ext groups associated with the standard exact sequences
\[
\xymatrix{
0\ar[r] &
\Omega^1_{\F_n}\ar[r]\ar@{=}[d] &
\Omega^1_{\F_n}(\log E)\ar[r]\ar[d] &
\mathcal{O}_E \ar[r]\ar[d] & 0 \\
0 \ar[r] &
\Omega^1_{\F_n}\ar[r] &
\Omega^1_{\F_n}(\log \sum D_i)\ar[r] &
\bigoplus \mathcal{O}_{D_i}\ar[r] & 0
}
\]
where $D_i$ are the toric divisors on $\F_n$,
and applying the triviality of the middle term of the bottom sequence
(see, e.g., \cite[\S8.1]{CLS}) and the isomorphism
$\Ext^j( \mathcal{O}_{D_i}, \mathcal{O}_{\F_n} ) \cong
H^{j-1}( D_i, \mathcal{O}_{\F_n}(D_i)|_{D_i})$.
\end{proof}


\begin{rema} \label{rem.expdim}
Lemma \ref{lem.smoothexpdim} is consistent with the
expected dimension of the Kontsevich space of 
maps $f:\bP^1 \to \mathcal{M}$.
Indeed, if $g:\bP^1\to [U/\PGL_2]$
denotes the composite map to $[U/\PGL_2]$ then the expected dimension is
$$
\deg(g^*K_{[U/\PGL_2]})-1=\frac{3}{2}h-1.
$$
\end{rema}

\section{Log general type surfaces}
\label{loggeneraltype}

We work over an algebraically closed field of characteristic different
from $2$.

\begin{prop} \label{prop:log}
Let $X$ be a smooth quartic Del Pezzo surface with lines
$D_1,\ldots,D_{16} \subset X$.
Then the pair $(X,(D_1,\ldots,D_{16}))$ is has
log canonical singularities and ample log canonical class.
\end{prop}
\begin{proof}
Computing in the Picard group, we find that
$$D_1+\cdots+D_{16}\equiv -4K_X$$
thus the log canonical class $K_X+D_1+\cdots+D_{16}$ is ample.  
Recall that at most two $D_i$ can be incident at any point $x\in X$---this
is straightforward from the classical realization of $X$ as the blow-up of
$\bP^2$ at five distict points with no three collinear.
Since each $D_i$ is smooth, the union $\cup_{i=1}^{16}$ is strict normal
crossings; thus the pair is log canonical.
\end{proof}

We would like to compactify the moduli space of quartic Del Pezzo surfaces,
considered as a moduli space of `stable log surfaces', in the sense of 
Koll\'ar, Shepherd-Barron, and Alexeev \cite{Alexeev,Hacking}.  
In positive characteristic the general
construction of moduli spaces of stable log surfaces
is not fully worked out, but the specific space we require can be obtained via
other techniques \cite{HKT}.

The definition of a family of stable log surfaces is still evolving; we
refer the reader to \cite{Kollar} for more detailed discussion.   
For our purposes, we may use the following restricted definition:
\begin{defi}
Let $B$ be a scheme of finite type over the base field.  
A {\em family of mildly singular stable log varieties} consists of
\begin{itemize}
\item{a scheme $\pi:\cX \ra B$ with $\pi$ flat, proper, and Gorenstein;}
\item{effective Cartier divisors $\cD_1,\ldots,\cD_r \subset \cX$
flat over $B$;}
\end{itemize}
satisfying the following:  For each closed point $b\in B$,
\begin{itemize}
\item{the pair 
$(\cX_b,(\cD_{1b},\ldots,\cD_{rb}))$ is semilog canonical;}
\item{$\omega_{\pi}(\cD_1+\cdots+\cD_r)|\cX_b$ is ample, where
$\omega_{\pi}$ is the relative dualizing sheaf.}
\end{itemize}
\end{defi}
Note that $\omega_{\pi}(\cD_1+\cdots+\cD_r)|\cX_b=
\omega_{\cX_b}(\cD_{1b}+\cdots+\cD_{rb})$ for each $b\in B$,
by standard properties of the dualizing sheaf.  

In general, stable varieties need not be Gorenstein and the
boundary components need not be Cartier, which is why we
describe these as `mild' singularities.  The mildness conditions
behave extremely well in families---both are open conditions:
If $\cX_b$ is Gorenstein then $\pi$ is Gorenstein over some
neighborhood of $b$; if $\cD_{j,b}$ is Cartier (and $\cD_j$
is flat over $B$) then $\cD_j$ is Cartier near $b$.  In practice,
this means that the mildly singular varieties are open in moduli
spaces of stable varieties.

Let $\wpM$ denote the connected
component of the moduli stack of stable log surfaces
$(X,(D_1,\ldots,D_{16}))$ containing the pairs
introduced in Proposition~\ref{prop:log}.
We recall some key properties:
\begin{itemize}
\item
$\wpM\simeq \oM_{0,5}$ \cite[Rem.~1.3, Thm.~10.19]{HKT},
with the smooth quartic Del Pezzos identified with $M_{0,5}$;
\item
the singularities of the fibers are `stably toric', obtained
by gluing together toroidal varieties along their boundaries
\cite[Thm.~1.1]{HKT}.
\end{itemize}

For our purposes, we enumerate the singular fibers over
the zero-dimensional and one-dimensional boundary strata of $\oM_{0,5}$.
This description is implicit in \cite[Rem.~1.3(5)]{HKT}:

\

\noindent {\bf one-dimensional}: The surface $X$ consists of six
components:
\begin{itemize}
\item{$X_1$, the minimal resolution of a quartic Del Pezzo surface
with a single node with conductor divisors
\begin{itemize}
\item{$B_{12}$ the exceptional divisor over the node;}
\item{$B_{1k},k=3,4,5,6$ the proper transforms of the lines
meeting the node;}
\end{itemize}}
\item{$X_2$, the blow up of a quadric surface at four coplanar points
$p_3,p_4,p_5,p_6$ with conductor divisors
\begin{itemize}
\item{$B_{12}$ the proper transform of the hyperplane sections;}
\item{$B_{2k}, k=3,4,5,6$ the exceptional divisors over the $p_k$;}
\end{itemize}}
\item{$X_k, k=3,4,5,6$ copies of the Hirzebruch surface $\bF_0$
with distinguished rulings
\begin{itemize}
\item{$B_{k1}\in |f_k|$ ;}
\item{$B_{k2}\in |f'_k|$;}
\end{itemize}
}
\end{itemize}
Note that $X_1$ and $X_2$ are in fact isomorphic.

The limits of the $16$ lines are
\begin{itemize}
\item{$D_{1k}$: union of the eight lines of $X_1$ {\em not } incident
to the node and the ruling $\in |f'_k|$ meeting it;}
\item{$D_{2k}$: union of the eight lines of $X_2$ {\em not } incident
to the node and the ruling $\in |f_k|$ meeting it.}
\end{itemize}

In particular, $X$ is D-semistable in the sense of Friedman \cite{Fr}; each
line $D_{ij}\subset X$ is cut out transversally.  Thus the singularities
are mild.

\begin{figure}
\centerline{
\includegraphics[scale=.6]{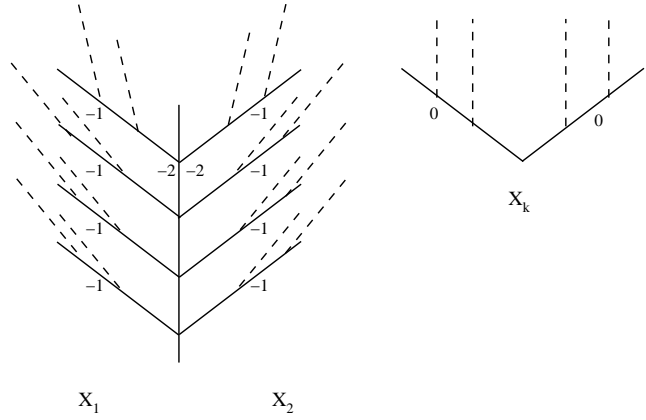}}
\caption{the components, with conductors solid and lines dashed}
\end{figure}

\

\noindent {\bf zero-dimensional}: The surface $X$ consists of $12$
components of two types:
\begin{itemize}
\item{four components $X_1,X_2,X_3,X_4$, each isomorphic to $\bP^2$ blown up at
four non-collinear points $\{x'_{i,i-1},x''_{i,i-1},x'_{i,i+i},x''_{i,i+1}\}, i \in \bZ/4\bZ$, with conductor divisor
\begin{itemize}
\item{the four exceptional divisors $B'_{i,i-1},B''_{i,i-1},B'_{i,i+1},B''_{i,i+1}$ in $X_i$;}
\item{the proper transforms $B_{i,i-1}$ and $B_{i,i+1}$ of the lines joining $\{x'_{i,i-1},x''_{i,i-1}\}$ and 
$\{x'_{i,i+1},x''_{i,i+1}\}$ respectively;}
\end{itemize}}
\item{eight components $X'_{12},X''_{12},X'_{23},X''_{23},
X'_{34},X''_{34},X'_{41},X''_{41}$ isomorphic to $\bF_0$, with conductor divisor
consisting of representatives from each ruling.}
\end{itemize}
Note that the components of the second type appear over the codimension-one boundary
points; however, we now have eight such components rather than four.  

We describe the limit of one of the $16$ lines, the others being defined symmetrically.
It has three irreducible components:
\begin{itemize}
\item{
In $X_1$, take the proper transform $D_{1,4',2'}$ of the line joining $\{x'_{1,4},x'_{1,2}\}$.}
\item{
In $X'_{41}$ take the ruling incident to $D_{1,4',2'}$ in one point.}
\item{
In $X'_{12}$ take the ruling incident to $D_{1,4',2'}$ in one point.}
\end{itemize}

\begin{figure}
\centerline{
\includegraphics[scale=.7]{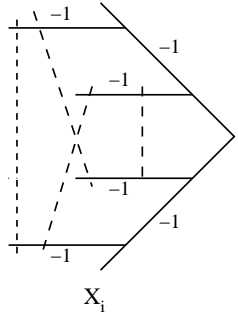}}
\caption{the four new components $X_1,\ldots,X_4$}
\end{figure}

\noindent
In particular, $X$ is D-semistable except at one point, where $X_1,X_2,X_3,X_4$ all intersect.  
Here $X$ is locally the cone over a cycle of four lines in $\bP^3$, e.g.,
$$\{x=z=0\}\cup \{y=z=0\} \cup \{y=x+z-1=0\} \cup \{x=y+z-1\}$$
which is a complete intersection
$$\{xy=z(x+y+z-1)=0\}$$
thus Gorenstein.  Each line avoids this point and is cut out transversally in $X$.  
These singularities are mild as well.

\

From this analysis, we deduce
\begin{prop}
The action of $W(\sD_5)$ on the moduli space of marked quartic Del Pezzo surfaces
extends naturally to a regular action on $\wpM$. The distinguished subgroup
$$(\bZ/2\bZ)^4=\ker(W(\sD_5) \ra \fS_5)$$
acts via automorphisms on the universal family.  The induced action of $\fS_5$
on $\wpM$ coincides with the standard relabelling action on $\oM_{0,5}$.  
\end{prop}

\begin{defi}
\label{defn.Mtilde}
Let $\wM=[\wpM/W(\sD_5)]$ denote the moduli stack of stable log surfaces
as above,
with {\em unordered} boundary divisors.
\end{defi}
This comes with a morphism
\begin{equation}
\label{eqn.Mtildetosomething}
\widetilde{\mathcal{M}}\cong
[\overline M_{0,5}/W(\sD_5)]\to
[\overline M_{0,5}/\fS_5]
\end{equation}
fitting into a fiber diagram with the morphism of classifying stacks
$BW(\sD_5)\to B\fS_5$.
In particular, $\wM$ is proper with coarse moduli space $\bP(1,2,3)$.
The morphism \eqref{eqn.Mtildetosomething} restricts to a morphism
\begin{equation}
\label{eqn.Msmtosomething}
\mathcal{M}^{\mathrm{sm}}\cong [M_{0,5}/W(\sD_5)]\to [M_{0,5}/\fS_5]\cong
[U^\circ/\PGL_2].
\end{equation}

\section{Comparison of moduli spaces}
\label{comparison}

We start by recalling a well-known geometric description of
torsors under symmetric and hyperoctahedral groups.
Let $n$ be a positive integer, and let $T$ be a scheme.
There is an equivalence
$$
\mathcal{C}ov^n(T) \cong B\fS_n(T)
$$ 
between the category of degree-$n$ \'etale covers of $T$, and the category of 
$\fS_n$-torsors over $T$, where $\fS_n$ denotes the symmetric group on
$n$ letters, sending
an \'etale cover $U\ra T$ to the open subscheme of
the $n$-fold fiber product $U^{\times n}$ over $T$ of
pairwise distinct points, and in the other direction associating to an $\fS_n$-torsor
$E\ra T$ the cover $E\times^{\fS_n}\{1,\ldots,n\}$.

There is a similar equivalence
\[
\mathcal{C}ov^{2,n}(T) \cong BW(\sB_n)(T)
\]
between the category of towers of \'etale coverings
$V\to U\to T$ with $V\to U$ of degree 2 and $U\to T$ of degree $n$ and the category of
torsors under the hyperoctahedral
group $W(\sB_n)=(\Z/2\Z)^n\rtimes \fS_n$.
For torsors under the index 2 subgroup $W(\sD_n)\subset W(\sB_n)$
(the type $\sD$ Weyl group)
there is a description
via towers $V\to U\to T$ together with a section of
$E/W(\mathsf D_n)\to T$, where $E$ denotes the associated
$W(\sB_n)$-torsor.
If $n$ is \emph{odd},
then the involution of $V\to U$
induces the involution of
the degree 2 cover $E/W(\mathsf D_n)\to T$.

The next result uses the above language to make
explicit the morphism \eqref{eqn.Msmtosomething}.

\begin{prop}
\label{prop.compatibility}
\emph{(i)} Let $T$ be a scheme and $T\to [U^\circ/\PGL_2]$
a morphism, corresponding to the cover $D\subset \bP(E)\to T$ for
a rank 2 vector bundle $E$ on $T$.
Then $D\to T$ is canonically identified with
the degree 5 cover associated with the composite
\[T\to [U^\circ/\PGL_2]\cong [M_{0,5}/\fS_5]\to B\fS_5.\]
\emph{(ii)} Let $T$ be a scheme and $\pi\colon \mathcal{X}\to T$ a smooth family of
Del Pezzo surfaces of degree 4, corresponding to a morphism
$T\to \mathcal{M}^{\mathrm{sm}}$.
Let 
$$
D\subset \bP(\pi_*(\mathcal{I}_{\mathcal{X}}(2)))\to T
$$ 
be the
spectral cover, parametrizing singular members of the pencils of
quadric hypersurfaces,
and $\tilde{D}\to D$ the degree 2 cover of singular quadric hypersurface with
family of rulings, cf.\ \cite[\S 3]{HT13}.
Then 
$$
\tilde{D}\to D\to T
$$ 
is canonically identified with the
tower of coverings associated with the composite morphism
\[T\to \mathcal{M}^{\mathrm{sm}}\cong [M_{0,5}/W(\sD_5)]\to BW(\sD_5).\]
\emph{(iii)} Let $T$ be a scheme,
$\pi\colon \cX\to T$ a smooth family of Del Pezzo surfaces of degree 4
with family of lines $\mathcal{L}\subset \cX$,
corresponding to a morphism
$T\to [U^\circ/\PGL_2]\cong
\mathcal{M}'^\circ\times_{\mathcal{M}^\circ}\mathcal{M}^{\mathrm{sm}}$.
Then with the notation of \emph{(ii)} we obtain a canonical identification
$\tilde D\cong D\times \Z/2\Z$
by labeling with $0$ the family of planes containing the linear span
of the vertex of a singular quadric and the chosen line of $\mathcal{L}$
and with $1$ the opposite family.
\end{prop}

\begin{proof}
The first assertion is clear. For the second assertion,
given a smooth Del Pezzo surface of degree 4,
each singular member of the pencil of quadric hypersurfaces
determines a partition of the set of 16 lines into
two disjoint sets each consisting of 4 pairs of intersecting lines.
Then $W(\sD_5)$ acts on the 5 such pairs of sets via the standard inclusion in the
hyperoctahedral group $(\Z/2\Z)^5\rtimes \fS_5$, cf.\ \cite{KST}, and the
assertion is clear.
For the third assertion, we merely recall that
the choice of family of lines $\mathcal{L}$
dictates a distinguished plane in every
singular member of the pencil of quadrics, namely the linear span with the
vertex.
\end{proof}

\begin{coro}
\label{cor.otherwaycompatible}
The morphism \eqref{eqn.scnofMsmtoUcircmodPGL2} fits into a fiber diagram
\[
\xymatrix{
[U^\circ/\PGL_2]\ar[r] \ar[d] & B\fS_5 \ar[d]
\\
\mathcal{M}^{\mathrm{sm}} \ar[r] & BW(\sD_5)
}
\]
with the right-hand morphism induced by
$\fS_5\subset W(\sD_5)$.
\end{coro}

\begin{prop}
\label{gerbeprop2}
\emph{(i)} There is a canonical \'etale representable morphism
\[
[ \overline{M}_{0,5} / \fS_5 ] \times _{ [ U / \PGL_2 ]}
\mathcal{M} \to
\widetilde{\mathcal{M}},
\]
extending the identity morphism on smooth families of
degree 4 Del Pezzo surfaces.

\noindent
\emph{(ii)} The morphism
$[ \overline{M}_{0,5} / \fS_5 ]\to
[ \overline{M}_{0,5} / \fS_5 ] \times _{ [ U / \PGL_2 ]}
\mathcal{M}$
determined by the section \eqref{eqn.scnofMtoUmodPGL2},
composed with the morphism in \emph{(i)}, is canonically 2-isomorphic to the morphism
\[
[ \overline{M}_{0,5} / \fS_5 ] \to
[\overline{M}_{0,5}/W(\sD_5)] \cong \widetilde{\mathcal{M}}
\]
coming from $\fS_5\subset W(\sD_5)$.
\end{prop}

\begin{proof}
For both statements, we use Nagata-Zariski purity,
(cf.\ \cite[Cor.\ X.3.3]{SGA1}) which tells us that the
restriction functor, from finite \'etale covers of a
regular locally noetherian scheme $X$ to
covers of a dense open subscheme $U$ is fully faithful, and
is an equivalence of categories if $X\smallsetminus U$ is of
codimension at least $2$.
The statement is equally valid if $X$ is an algebraic stack.
For (i), by the fiber diagram mentioned just after Definition \ref{defn.Mtilde}
it suffices to show that the tautological $W(\sD_5)$-torsor over
$\mathcal{M}^{\mathrm{sm}}\cong [M_{0,5}/W(\sD_5)]$
extends to a $W(\sD_5)$-torsor over the complement of a codimension 2
substack of
$[ \overline{M}_{0,5} / \fS_5 ] \times _{ [ U / \PGL_2 ]}
\mathcal{M}$.
We consider the substack of Del Pezzo surfaces with at most one
$\sA_1$-singularity.
Then the recipe to produce a log canonical model is to
blow up the singular locus described in Corollary \ref{gerbecor},
then to blow up the $(-1)$-curves in the fibers meeting the
relative singular locus.
The proof of (ii) is similar, using
Corollary \ref{cor.otherwaycompatible}.
\end{proof}

By Proposition \ref{gerbeprop2}(i) we have a 2-commutative diagram with
cartesian square.
\begin{equation}
\begin{split}
\label{eqn.diagr}
\xymatrix@C=10pt{
& [ \overline{M}_{0,5} / \fS_5 ] \times _{ [ U / \PGL_2 ]}
\mathcal{M}
\ar[r] \ar[dl] \ar[d] &
\mathcal M \ar[d] \\
\widetilde{\mathcal{M}}\cong [\overline{M}_{0,5}/W(\sD_5)]
\ar[r] &
[ \overline{M}_{0,5} / \fS_5 ] \ar[r] &
[ U / \PGL_2 ]
}
\end{split}
\end{equation}
Given a family of Del Pezzo surfaces of degree 4 over a regular base
whose discriminant divisor is a multiple of 2,
Proposition \ref{gerbeprop2}(i) supplies a family of log general type surfaces,
which we call the \emph{associated family} of log general type surfaces.

\begin{prop}
\label{prop.birational}
Let $R$ be a discrete valuation ring over $k$ with algebraically closed
residue field, and let
$\cX\to \Spec(R)$ and $\cX'\to \Spec(R)$ be generically smooth
families of Del Pezzo surfaces of degree 4, each with
central fiber having a single $\sA_1$-singularity and with
discriminant of valuation $2$.
Let
\[\widetilde{\cX}=\Bl_{4\,\mathrm{lines}}(\Bl_{\cX^{\mathrm{sing}}}(\cX))\]
where we recognize by the hypotheses that $\cX$ has one singular point,
an ordinary double point, which is the center of the first blow-up,
and where the second blow-up is along the proper transform of
the four lines in the central fiber meeting $\cX^{\mathrm{sing}}$.
Define $\widetilde{\cX}'$ similarly, and let
\[\varphi\colon \widetilde{\cX}\to \widetilde{\cX}'\]
be an isomorphism over $\Spec(R)$.
Then the rational map
\[ \cX \dashrightarrow \cX' \]
induced by $\varphi$ is a morphism when $\varphi$ sends exceptional
divisors to exceptional divisors and is only a rational map otherwise.
\end{prop}

\begin{proof}
Under the composite blow-down morphism $\widetilde{\cX}\to \cX$ the
pre-image of the complement of the four lines meeting $\cX^{\mathrm{sing}}$
is the complement of the exceptional divisors in $\widetilde{\cX}$.
If $\varphi$ sends exceptional divisors to exceptional divisors then
$\varphi$ induces an isomorphism
\[\cX\smallsetminus\{4\ \mathrm{lines}\}\to
\cX'\smallsetminus\{4\ \mathrm{lines}\}.\]
Since $\cX$ is normal, this extends to an isomorphism $\cX\to \cX'$.
Conversely, an isomorphism $\cX\to \cX'$ induces an isomorphism of
blow-ups sending exceptional components to exceptional components.
\end{proof}

\begin{prop}
\label{prop.mu2action}
Let $R$ be a strictly henselian discrete valuation ring over $k$
with $\varpi\in R$ a uniformizer, and let
\[ D\subset \bP^1_R \]
be a divisor, with $D$ isomorphic over $\Spec(R)$ to
\[ \Spec(R(\sqrt{\varpi}))\amalg \coprod_{i=1}^3 \Spec(R). \]
Let
\[ \iota\colon \bP^1_R\to \bP^2_R, \]
denote the Veronese embedding, and define
\[ \cX_0=\Bl_{\iota(D)}(\bP^2_R),\qquad\qquad
\cX=\Spec(R(\sqrt{\varpi}))\times_{\Spec(R)}\cX_0. \]
Define $\widetilde{\cX}$ as in Proposition \ref{prop.birational},
i.e., as the blow-up along four lines of the blow-up of the singular point
of $\cX$.
Notice that $\widetilde{\cX}$ has a unique exceptional divisors
isomorphic to the minimal resolution of an $\sA_1$-singular
Del Pezzo surface of degree $4$;
we call this the Del Pezzo exceptional divisor.
Let $\tau\colon R(\sqrt{\varpi})\to R(\sqrt{\varpi})$ be the nontrivial
Galois automorphism.
Then the automorphism $\Spec(\tau)\times \mathrm{id}_{\cX_0}$
of $\cX$
induces an automorphism of $\widetilde{\cX}$ which
maps the Del Pezzo exceptional divisor nontrivially to itself.
\end{prop}

\begin{proof}
Since the blow-ups have Galos-invariant centers,
the automorphism of $\cX$ induces an automorphism of $\widetilde{\cX}$
mapping exceptional divisors to exceptional divisors, hence the
Del Pezzo exceptional divisor to itself.
The Del Pezzo exceptional divisor is birational to the
projectivized normal cone to $\cX^{\mathrm{sing}}$.
A computation in formal local coordinates establishes the result.
\end{proof}

\section{Maps to $\mathcal{M}$ from covers of spectral curves}
\label{maps}

\begin{theo}
\label{thm.maptoM}
Let $k$ be a field of characteristic different from $2$. 
There is an equivalence of fibered categories over $k$-schemes between
\begin{itemize}
\item
families $\cX\to T\times \bP^1$
of degree 4 Del Pezzo surfaces (flat, proper, finitely presentated, with
fibers having $\mathsf{ADE}$-singularities) satisfying $(\mathrm{G1})$ and $(\mathrm{G2})'$
over all geometric points
of $T$, and
\item the fibered category whose fiber over $T$ consists of
$D\subset \F\to T\times \bP^1$ with $\F\to T\times \bP^1$ a $\bP^1$-bundle,
$D\to T$ smooth with $D\to T\times \bP^1$ finite flat of degree 5 and
satisfying $(\mathrm{G1})'$ and $(\mathrm{G2})'$
over all geometric points of $T$, together with
a section of the $2$-torsion of the relative Jacobian
$J(D/T)[2]\to T$,
\end{itemize}
given by the spectral cover construction together with the
fiberwise double cover of families of planes
in the singular quadric hypersurfaces.
\end{theo}

\begin{proof}
The functor in the forward direction is given as follows.
To a family $\pi\colon \cX\to T\times \bP^1$ we associate the
spectral cover
\[
D\subset \bP(\pi_*(\mathcal{I}_{\cX}(2)))\to T\times \bP^1
\]
which parametrizes singular members of the pencils of quadric hypersurfaces
of the fibers of $\pi$, together with the
section of $J(D/T)[2]\to T$ corresponding to the
double cover $\tilde D\to D$ of families of planes in the
singular quadrics.
Compatiblity of the construction with base change is obvious, so we have
a functor between the fibered categories.

This functor is a morphism between algebraic stacks that are \'etale over
$\Hom_{(\mathrm{G1})'}(\bP^1,[U/\PGL_2])$.
So the morphism is \'etale, and to verify that it is an isomorphism
it suffices to show that each geometric fiber consist of a single point,
with trivial stabilizer.
For this we may assume that $k$ is algebraically closed and that we are
given $D\subset \F\to \bP^1$ with $\F$ a Hirzebruch surface,
$D$ a nonsingular irreducible curve and $D\to \bP^1$ finite of degree 5 and simply ramified
over a divisor $\Delta\subset \bP^1$,
together with a $2$-torsion element of the Jacobian $J(D)$.
Let $\bP^1\to [U/\PGL_2]$ and $\tilde D\to D$ be the corresponding morphism,
respectively, cover.
Let
\[\mathcal{C}=[ \overline{M}_{0,5} / \fS_5 ] \times _{ [ U / \PGL_2 ]}\bP^1,\]
so we have the solid arrows of a diagram
which extends the diagram \eqref{eqn.diagr}:
\begin{equation}
\begin{split}
\label{eqn.diagr2}
\xymatrix@C=12pt{
\mathcal{C}
\ar[rrrr] \ar[ddrr] \ar@{-->}[dddr]
\ar@{-->}[ddr]
\ar@{-->}[drr]
&&&&
\bP^1\ar[ddl]
\\
&& [ \overline{M}_{0,5} / \fS_5 ] \times _{ [ U / \PGL_2 ]}
\mathcal{M}
\ar[r] \ar[dl]|(0.41)\hole \ar[d] &
\mathcal M \ar[d] \\
&\widetilde{\mathcal{M}}
\ar[r] \ar[d] &
[ \overline{M}_{0,5} / \fS_5 ] \ar[r] \ar[d] &
[ U / \PGL_2 ]
\\
& BW(\sD_5) \ar[r] & B\fS_5
}
\end{split}
\end{equation}
In the diagram, the lower square is a fiber square,
as mentioned just after Definition \ref{defn.Mtilde}.

By the discussion in the proof
of Lemma \ref{M05modS5}, $\mathcal{C}$
is a root stack over $\bP^1$ along $\Delta$.
So if we let $\mathcal{D}$ be the normalization of
$\mathcal{C}\times_{\bP^1}D$, then $\mathcal{D}$ is \'etale cover $\mathcal{C}$.
Setting $\tilde{\mathcal{D}}=\tilde D\times_D\mathcal{D}$, we have a tower
\begin{equation}
\label{eqn.tower}
\tilde{\mathcal{D}}\to\mathcal{D}\to \mathcal{C}
\end{equation}
The dashed arrows in the diagram are obtained by examining
the tower \eqref{eqn.tower}
near an orbifold point of $\mathcal{C}$.
If we pass to the henselization of the local ring of $\bP^1$ at a point of
$\Delta$ then we obtain, by base change,
$[\Spec(k[t]_{(t)}^{h})/\mu_2]\to \mathcal{C}$.
The base change of $\mathcal{D}$ is
\begin{equation}
\label{eqn.particularcover}
\Spec(k[t]_{(t)}^{h}) \amalg \coprod_{i=1}^3
[\Spec(k[t]_{(t)}^{h})/\mu_2],
\end{equation}
corresponding to the morphism $[\Spec(k[t]_{(t)}^{h})/\mu_2]\to B\fS_5$
given by a transposition in $\fS_5$.
Since $\tilde{\mathcal{D}}\to \mathcal{D}$ is obtained by base change from
$\tilde D\to D$, the only possibility is that
$\tilde{\mathcal{D}}\to \mathcal{D}$ base-changes to a trivial cover of
\eqref{eqn.particularcover}.
In particular, if we let $\mathcal{E}$ be the open substack
of $\tilde{\mathcal{D}}^{\times n}$ as in the description in
Section \ref{comparison},
then the base change of $\mathcal{E}/W(\sD_5)\to \mathcal{C}$ is
also a trivial cover, so since $\mathcal{C}$ is an orbifold $\bP^1$,
the cover $\mathcal{E}/W(\sD_5)\to \mathcal{C}$ must be globally trivial.
So there is a canonical $W(\sD_5)$-torsor over $\mathcal{C}$ corresponding to
the tower \eqref{eqn.tower}, and hence a bottom dashed arrow in the diagram,
determined up to unique 2-isomorphism.
(Recall, the involution of $\tilde{\mathcal{D}}$ over $\mathcal{D}$ switches
the two sections of $\mathcal{E}/W(\sD_5)\to \mathcal{C}$, so the
$W(\sD_5)$-torsor structure is canonical.)
The next dashed arrow is obtained using the universal property of a
fiber diagram, i.e., we have a canonically defined family of
log general type surfaces $ \epsilon \colon \mathcal Y \to \mathcal{C}$.
To give the final dashed arrow is equivalent to giving a section of
\begin{equation}
\label{eqn.etaleoverC}
\mathcal{C}
\times_{\widetilde{\mathcal{M}}}
( [ \overline{M}_{0,5} / \fS_5 ] \times _{ [ U / \PGL_2 ]}
\mathcal{M} )
\to
\mathcal{C}.
\end{equation}
This is an \'etale morphism, an isomorphism away from the orbifold points
of $\mathcal{C}$, and of degree $2$ over the
orbifold points of $\mathcal{C}$.
By Proposition \ref{prop.birational}, the fiber of \eqref{eqn.etaleoverC}
over an orbifold point of $\mathcal{C}$ is identified with the 2-element set of
components of the fiber of $\epsilon$, isomorphic to a resolution of a
singular Del Pezzo surface of degree 4; we call these the
Del Pezzo components.
Combining Propositions \ref{gerbeprop2}(ii) and \ref{prop.mu2action}, we see that
$\mu_2$ acts trivially on the set of Del Pezzo components, and
the $\mu_2$-action on one of the Del Pezzo components is trivial and on the
other is nontrivial.
So the morphism \eqref{eqn.etaleoverC} admits sections,
and a section is specified uniquely by dictating the choice of
Del Pezzo component that is acted upon trivially by $\mu_2$ at each
orbifold point.
With this uniquely specified section we have, canonically determined, a family
$\pi'\colon \cX'\to \mathcal{C}$ of Del Pezzo surfaces
and an isomorphism of the associated family of log general type surfaces
with $\mathcal Y$.
Since the $\mu_2$-action on the fiber of $\pi'$
at any orbifold point of $\mathcal{C}$ is trivial,
$\pi'$ is obtained by base change from a family
$\pi\colon \cX\to \bP^1$ of Del Pezzo surfaces, defined up to a
unique isomorphism.
The spectral curve is $D\subset \F\to \bP^1$, and the double cover associated to the
families of planes in singular quadric hypersurfaces is $\tilde D\to D$.
So we have shown that morphism of algebraic stacks, described in the statement
of the theorem, is surjective.
Noting the uniqueness (up to canonical 2-isomorphisms) of the dashed arrows
in \eqref{eqn.diagr2} and the uniqueness of the family $\pi$ in the last step,
we obtain that the morphism is an isomorphism.
\end{proof}

We apply this description to enumerate the components of general families
of given height.

\begin{theo}
\label{thm.main}
Fix a height $h$, even and positive.
The space of families of height $h$ satisfying Conditions $(\mathrm{G1})$ and $(\mathrm{G2})$
is empty when $h\le 6$ and for $h\ge 8$ consists of:
\begin{itemize}
\item[(i)] two components when $h=8$ or $h=10$;
\item[(ii)] one component when $h\ge 12$.
\end{itemize}
When $h=8$, the spectral curve is hyperelliptic,
hence the 2-torsion of the Jacobian is partitioned according to the
minimal number of Weierstrass points (minus the appropriate multiple
of the $g^1_2$) in the canonical represenation.
The two components with monodromy $W(\sD_5)$ correspond to sums of
two, respectively four Weierstrass points.
When $h=10$, the spectral curve is a plane quintic curve, hence
comes with a natural theta characteristic and an associated quadratic form on the
2-torsion of the Jacobian.
In this case the two components with monodromy $W(\sD_5)$
correspond to the two values of the quadratic form on
the nonzero points of the 2-torsion of the Jacobian of $D$.
\end{theo}

\begin{proof}
By Theorem \ref{thm.maptoM}, we may consider spectral curves in
Hirzebruch surfaces with
choice of 2-torsion in the Jacobian.
By the relative smoothness assertion of Lemma \ref{lem.smoothexpdim},
we may restrict to case of the Hirzebruch surfaces $\F_0=\bP^1\times \bP^1$
and $\F_1$.

Assume $h\ge 12$.
By the formula \eqref{eqnD},
for $h$ divisible by 4 we have
$D\subset \bP^1\times \bP^1$ of bidegree $(h/4,5)$,
and for $h$ congruent to 2 modulo 4 we have
$D\subset \F_1$ with $[D]=((h+10)/4)f+5\xi$.
In each case Lemma \ref{lem.monodromy} implies that the monodromy
action on the 2-torsion in the Jacobian of the spectral curve is the
full symplectic group.
So the space of pairs $(D,\tilde D\to D)$ with $D$ as above
and $\tilde D\to D$ a nontrivial unramified degree 2 cover consists of a
single component.

When $h=10$ we have $D\subset \F_1$ with $[D]=5f+5\xi$,
i.e., if we identify $\F_1$ with the blow-up of a point in the projective plane
then $D$ is the pre-image of a smooth quintic curve not passing through the
point that is blown up.
In this case the restriction of $\cO_{\bP^2}(1)$
is a theta characteristic, so the monodromy group is cut down to
$\mathrm{O}(H^2(D,\Z/2\Z),q)$ where $q$ is the corresponding quadratic form.
When $h=8$, the spectral curve is of bidegree $(2,5)$ in
$\bP^1\times \bP^1$, and Example \ref{examhyperell} furnishes the
complete description.
\end{proof}

\section{Examples of families of low height}
\label{sect:low-height}

In this section, we provide examples, 
in height 8 and 10, of distinct families
of quartic Del Pezzo surface fibrations 
$\pi:\cX\ra \bP^1$, of expected dimension and maximal monodromy.

\

\noindent
{\bf Height 8}:
We recall the construction from \cite[Remark 15]{HT13}. 
Let 
$$
\cX\subset \bP^1\times \bP^5
$$ 
be a complete intersection of a form of bidegree $(1,1)$ and two forms of bidegree $(0,2)$. 
The projection to the second factor $\cX\ra \cY\subset \bP^5$ gives a 
complete intersection of two quadrics, and the quartic 
Del Pezzo surface fibration $\pi: \cX\ra \bP^1$ 
corresponds to a pencil of hyperplane sections, 
with base locus a smooth curve $E$ of genus 1. In turn, 
projection from a line $\ell$ in $\cY$ is  
the blowup of $\bP^3$ in a quintic curve of genus 2, so that 
$$
\IJ(\cX)\simeq\JJ(C)\times E.
$$

\

\noindent
{\bf Height 8}:  Consider the vector bundle
$$E=\cO_{\bP^1 \times \bP^1}^2 \oplus \cO_{\bP^1 \times \bP^1}(-1,-1)$$
and its projectivization $\phi:\bP(E) \ra \bP^1 \times \bP^1$.  
Let $\phi_i:\bP(E) \ra \bP^1$ denote the resulting projections.
Let $W\subset \bP(E)$ be a conic fibration corresponding to $\cO_{\bP(E)}(2)\otimes \phi_2^*\cO_{\bP^1}(1)$,
given by a section of 
$$
\Sym^2(E^{\vee})\otimes \cO_{\bP^1 \times \bP^1}(0,1).
$$  
Such a section corresponds to a symmetric matrix of forms
$$A:=\left(\begin{matrix} A_{11} & A_{12}  & A_{13}   \\
		       A_{12}& A_{22}&  A_{23} \\
		       A_{13} & A_{23} & A_{33} \end{matrix} \right)
$$
where $A_{11},A_{12},A_{22}$ have bidegree $(0,1)$, $A_{13}$ and
$A_{23}$ have bidegree $(1,2)$, and $A_{33}$ has bidegree $(2,3)$.
Note that 
$$
h^0(\Sym^2(E^{\vee})\otimes \cO_{\bP^1 \times \bP^1}(0,1))=30
$$
and  $h^0(\End(E))=13$,
so $\dim \Aut(\bP(E))=18$ and hence $W$ depends on $11$ parameters.  

The discriminant
curve $D \subset \bP^1 \times \bP^1$ has bidegree $(2,5)$, thus is hyperelliptic of genus four.
Thus the projection $\phi_1:W \ra \bP^1$ has fibers isomorphic to conic bundles with five degenerate
fibers, whence cubic surfaces.  
Consider the subvariety 
$$
W'=\bP(\cO_{\bP^1 \times \bP^1}^2)\cap W \subset \bP^1_1 \times \bP^1_2 \times \bP^1_3;
$$
we order so that $\phi_1$ and $\phi_2$ map to $\bP^1_1$ and $\bP^1_2$ and
$\bP^1_3$ is the fiber of $\bP(\cO^2_{\bP^1\times\bP^1})\ra \bP^1_1 \times \bP^1_2$.  We regard $W'$ as a bisection of 
$W\ra \bP^1_1 \times \bP^1_2$.
A degree computation shows that $W'$ is the preimage of a 
curve $B\subset \bP^1_2 \times \bP^1_3$ of bidegree $(1,2)$.  

Now $f_D:D \ra \bP^1_2$ and $f_B:B \ra \bP^1_2$ both have degree two with ten
and two branch points respectively. 
\begin{lemm}  The branch locus of $f_B$ is contained in the branch locus of $f_D$. 
\end{lemm}
\begin{proof}
Recall that $D=\{\det(A)=0\}.$
The branch locus of $f_B$ is given by $\{A_{11}A_{22}-A_{12}^2=0\}$.
Note that
$$\det(A)\equiv -A_{11}A_{23}^2-A_{22}A_{13}^2+2A_{12}A_{23}A_{13}\pmod{A_{11}A_{22}-A_{12}^2}.$$
Let $u$ and $v$ be homogeneous coordinates of $\bP^1_1$; write
$$A_{13}=A_{13}'u+A_{13}''v, \quad A_{23}=A_{23}'u+A_{23}''$$
and expand
$$-A_{11}A_{23}^2-A_{22}A_{13}^2+2A_{12}A_{23}A_{13}=a u^2 + 2b uv + c v^2, \quad a,b,c \in \Gamma(\cO_{\bP^1_2}(5)).$$
The branch locus of $f_D$ equals $\{ac-b^2=0\}$ modulo $A_{11}A_{22}-A_{12}^2$.    
A direct computation shows
\begin{eqnarray*}
ac-b^2&=&(A_{11}A_{22}-A_{12}^2)(A_{13}'A_{23}''-A_{13}''A_{23}')^2 \\
      &=&\det \left(\begin{matrix}A_{11} &  A_{12} \\ A_{12} & A_{22} \end{matrix} \right)
      \det \left(\begin{matrix}A'_{13} &  A''_{13} \\ A'_{23} & A''_{23} \end{matrix} \right)^2.
\end{eqnarray*}
\end{proof}

This covering data determines the intermediate Jacobian $\IJ(W)$:  The general theory of conic bundles over 
rational surfaces \cite{beauville77} implies $\IJ(W)=\Prym(\tilde{D}\ra D)$ for the \'etale double cover
arising from the irreducible components of the singular conics over the discriminant $D$.  In our
situation, $\tilde{D}$ is the normalization of the fiber product $D\times_{\bP^1_2} B$.  
It follows that 
$$\Prym(\tilde{D} \ra D)=\JJ(C),$$
where $C$ is a double curve of $\bP^1_2$ branched over the complement of
the branch locus of $f_B$ in the branch locus of $f_D$ \cite[p.~303]{acgh}.  
In particular, $C$ is hyperelliptic of genus three.  

Fixing $p \in \bP^1_1$, we find that 
$$\bP(E)|_{\{p\} \times \bP^1_2} \simeq \bP(\cO_{\bP^1}^2 \oplus \cO_{\bP^1}(-1)) \subset \bP^3 \times \bP^1_2,$$
i.e., the planes containing a fixed line $\ell \subset \bP^3$.  Morever
$W|_{\{p\}\times \bP^1_2}$
may be interpreted as one of the cubic surfaces containing this line, and $W'|_{\{p\}\times \bP^1_2}$
as the bisection induced by $\ell$.  Blowing down $\ell$ in the generic fiber of $\phi_1$, we obtain a fibration
in quartic Del Pezzo surfaces 
$$
\pi:\cX \ra \bP^1_1.
$$  

\

\noindent
{\bf Height 10}: 
Consider 
$$
\cX\subset \bP^1\times \bP^4, 
$$
given as a complete intersection of forms of bidegree $(0, 2)$ and $(1,2)$. Then 
$\cX=\Bl_C(Q)$, the blowup of a smooth quadric $Q\subset \bP^4$ in a smooth 
canonical curve $C$ of genus 5, 
the base locus of a pencil of quadrics on $Q$ defining $\pi: \cX\ra \bP^1$. Thus
$$
\IJ(\cX)\simeq \JJ(C).
$$

\

\noindent
{\bf Height 10}:  Fix a smooth cubic threefold $\cW\subset \bP^4$ and a conic curve $Q\subset \cW$;
this data depends on $14$ parameters.
Let $\ell \subset \cW$ denote the line residual to $Q$ in $P=\mathrm{span}(Q)$ and set $\ell \cap Q=\{w_1,w_2\}$.
Consider the pencil of hyperplane sections of $\cW$ associated with $P$, which induces a cubic
surface fibration
$$\Bl_{\ell \cup Q} (\cW) \dashrightarrow \bP^1.$$
The total space has two ordinary threefold singularities over $w_1$ and $w_2$; these
are in the fibers associated with the tangent hyperplanes to $\cW$ at $w_1$ and $w_2$.  
A small resolution $\widetilde{\cW}\ra \Bl_{\ell \cup Q}$ 
may be obtained by blowing up $Q$ and then $\ell$.  
Each fiber contains of $\varpi:\widetilde{\cW} \ra \bP^1$ contains $\ell$ as well as $Q$.  
Blowing down the exceptional divisor $\ell\times \bP^1\subset \widetilde{\cW}$ yields a fibration
$$\pi:\cX \ra \bP^1$$
in quartic Del Pezzo surfaces.  

The intermediate Jacobian $\IJ(\cX)\simeq \IJ(\cW)$ has numerous Prym interpretations:  
For each line $\ell \subset \cW$
projecting from $\ell$ induces a conic bundle structure 
$$\Bl_{\ell}(\cW) \ra \bP^2$$
with discriminant curve a plane quintic $D$; 
$$
\IJ(\cW) \simeq \Prym(\tilde{D} \ra D).
$$ 
The conic $Q$ corresponds
to fixing a point $q\in \bP^2$, and the spectral cover $D\ra \bP^1$ arises from projective from $q$.

To get explicit equations for $\cX$, consider the vector bundle
$$V=\cO_{\bP^1}(-2) \oplus \cO_{\bP^1}(-1)^3 + \cO_{\bP^1}$$
and the associated projective bundle $\bP(V)$ with relative hyperplane class 
$\eta$.  Let $h$ be the pull back of the hyperplane class from $\bP^1$.  
Let $\cX$ be a complete intersection of divisors in
$\bP(V)$ of degree $2\eta-2h$ and $2\eta-h$.  
The canonical class of $\bP(V)$ is $-5\eta+3h$ so the
canonical class of $\cX$ is $-\eta$.

We have a natural inclusion of $V \subset \cO_{\bP^1}^{10}$ inducing a morphism
$$\cX \ra \bP^9$$
with image $\cY$ singular at the image of the summand $\cO_{\bP^1}\subset V$.  This is a singular Fano threefold of genus
eight; the smooth varieties in this class arise as codimension-five linear sections of $\Gr(2,6)$.

\appendix

\section{Alternative characterization of height}
\label{appendix}

Working over $\bC$, we consider a flat family $\pi\colon \cX\to \bP^1$ of
degree 4 Del Pezzo surfaces
with $\mathsf{ADE}$-singularities.
In Definition \ref{height} we defined the height $h(\cX)$ as the degree
of a vector bundle on $\bP^1$.
If we assume that family is generically smooth with square-free
discriminant, then $\cX$ is a smooth projective
threefold.
In this case in \cite{HT13} the height is
defined as a triple intersection number on $\cX$.
Here we show that these two definitions agree.

\begin{prop}
Let $\pi\colon \cX\to \bP^1_{\bC}$ be a generically smooth family of
degree 4 Del Pezzo surfaces with square-free discriminant.
Then
\[
\int_{\cX} c_1(\omega_\pi)^3 = -2\deg (\pi_*\omega_\pi^{-1}).
\]
\end{prop}

\begin{proof}
Applying Lemma \ref{lem:easyintersectiontheory}, we have
\begin{align*}
[\cX]&=c_2(\cO_{\bP((\pi_*\omega_\pi^{-1})^\vee)}(2)\otimes \pi^*\pi_*(\cI_{\cX}(2))) \\
&=4c_1(\cO_{\bP((\pi_*\omega_\pi^{-1})^\vee)}(1))^2
-2\pi^*c_1(\pi_*(\cI_{\cX}(2)))c_1(\cO_{\bP((\pi_*\omega_\pi^{-1})^\vee)}(1)).
\end{align*}
Therefore,
\begin{align*}
\int_\cX c_1(\omega_\pi^{-1})^3&=
\int_{\bP((\pi_*\omega_\pi^{-1})^\vee)} [\cX]\cdot c_1(\cO_{\bP((\pi_*\omega_\pi^{-1})^\vee)}(1))^3\\
&=4\int_{\bP((\pi_*\omega_\pi^{-1})^\vee)}
c_1(\cO_{\bP((\pi_*\omega_\pi^{-1})^\vee)}(1))^5\\
&\qquad
-2\int_{\bP((\pi_*\omega_\pi^{-1})^\vee)}
\pi^*c_1(\pi_*(\cI_{\cX}(2)))c_1(\cO_{\bP((\pi_*\omega_\pi^{-1})^\vee)}(1))^4\\
&=4\int_{\bP^1} c_1(\pi_*\omega_\pi^{-1})
-2\int_{\bP^1} c_1(\pi_*(\cI_{\cX}(2)))\\
&=2\int_{\bP^1} c_1(\pi_*\omega_\pi^{-1}).
\end{align*}
where at the last step we have used
\eqref{degequalsdeg}.
\end{proof}

\end{document}